\documentclass[10pt,a4paper,oneside,english,reqno]{amsart}
\usepackage[T1]{fontenc}
\usepackage[latin9]{inputenc}
\usepackage{color}
\usepackage{babel}
\usepackage{mathrsfs}
\usepackage{amstext}
\usepackage{amsthm}
\usepackage{amssymb}
\usepackage[unicode=true,pdfusetitle,
 bookmarks=true,bookmarksnumbered=false,bookmarksopen=false,
 breaklinks=false,pdfborder={0 0 1},backref=false,colorlinks=false]
 {hyperref}

\makeatletter

\pdfpageheight\paperheight
\pdfpagewidth\paperwidth

\newcommand{\noun}[1]{\textsc{#1}}

\numberwithin{equation}{section}
\numberwithin{figure}{section}
  \theoremstyle{definition}
  \newtheorem{defn}{\protect\definitionname}
\theoremstyle{plain}
\newtheorem{thm}{\protect\theoremname}
  \theoremstyle{remark}
  \newtheorem{rem}{\protect\remarkname}
 \theoremstyle{definition}
  \newtheorem{example}{\protect\examplename}
  \theoremstyle{plain}
  \newtheorem{cor}{\protect\corollaryname}
  \theoremstyle{plain}
  \newtheorem{lem}{\protect\lemmaname}
  \theoremstyle{plain}
  \newtheorem{fact}{\protect\factname}

\@ifundefined{date}{}{\date{}}
\makeatother

  \providecommand{\definitionname}{Definition}
  \providecommand{\examplename}{Example}
  \providecommand{\factname}{Fact}
  \providecommand{\lemmaname}{Lemma}
  \providecommand{\remarkname}{Remark}
\providecommand{\corollaryname}{Corollary}
\providecommand{\theoremname}{Theorem}

\begin{document}

\title{\textmd{ON THE CONSERVATIVE PASTING LEMMA}}

\author{Pedro Teixeira}
\begin{abstract}
Several perturbation tools are established in the volume preserving
setting allowing for the pasting, extension, localized smoothing and
local linearization of vector fields. The pasting and the local linearization
hold in all classes of regularity ranging from $C^{1}$ to $C^{\infty}$
(Hölder included). For diffeomorphisms, a conservative linearized
version of Franks lemma is proved in the $C^{r,\alpha}$ ($r\in\mathbb{Z}^{+}$,
$0<\alpha<1$) and $C^{\infty}$ settings, the resulting diffeomorphism
having the same regularity as the original one.
\end{abstract}


\maketitle

\section{Introduction}

\hypertarget{se1.1}{}

\subsection{Continuous-time dynamics}

One of the basic problems in conservative continuous-time dynamics
is the following:

\emph{How may a local $C^{r}$-perturbation of a divergence-free vector
field be extended to a global one?}

More precisely (and always in the conservative setting), given a $C^{r}$
vector field $X$ on a closed connected manifold $M$ and a $C^{r}$-perturbation
$Y$ of the restriction of $X$ to an open set $U$, is it possible
to find a $C^{r}$-perturbation $Z$ of $X$ that still coincides
with $Y$ in a slightly smaller set, say in any chosen compact set
$K\subset U$? In the non-conservative context the solution is trivial,
$Y$ can be glued with $X$ using a suitable partition of unity, i.e.
we let $\widehat{Z}=\xi Y+(1-\xi)X$ in $U$ and $\widehat{Z}=X$
in $U^{c}$ where the smooth function $\xi$ equals $1$ in a neighbourhood
of $K$ and $0$ in neighbourhood of $U^{c}$. Clearly $\widehat{Z}$
is $C^{r}$-close to $X$ if $Y$ is $C^{r}$-close to $X$ in $U$
and the problem is solved. 

In the conservative setting the situation is more delicate, for $\widehat{Z}$
constructed as above fails in general to be divergence-free in the
transition ``annulus'' $\varOmega$ i.e. in the set where $0<\xi<1$.
One obvious way to tackle this difficulty is trying to find a $C^{r}$
vector field $v$ supported in $\overline{\varOmega}$ whose divergence
equals that of $\widehat{Z}$ and then set $Z=\widehat{Z}-v$, thus
canceling the divergence. Provided $v$ can be found $C^{r}$-small
if $Y-X|_{U}$ is $C^{r}$-small, the question is solved.

The problem is that, in the conservative setting, an obstruction of
topological nature may hinder the above procedure: the interplay between
the divergence theorem and connected cobordism. To simplify the explanation,
all manifolds referred to below are assumed to be compact, connected,
orientable and smooth (besides second countable and Hausdorff). Let
$M$, $U$ and $K$ be as above. We start by observing that $K$ may
contain a closed $(n-1)$-submanifold $\gamma$ which is the boundary
of \emph{no} $n$-submanifold contained in $U$. In this case, the
perturbation $Y$ of the restriction of $X$ to $U$ may change the
original flux across $\gamma$ (see Example \hyperlink{Ex1}{1} below).
But simultaneously, there might exist another closed $(n-1)$-submanifold
$\gamma'$, now contained in $U^{c}$, that together with $\gamma$
constitutes the boundary of an $n$-submanifold $W$. Note that the
divergence canceling procedure described above assures that $Z=X$
in $U^{c}$, thus the original flux of $X$ across $\gamma'$ is kept
unchanged in $Z$. As a consequence, the flux of $Z$ across the cobordant
manifolds $\gamma$ and $\gamma'$ will be distinct, thus implying
(by the divergence theorem) that the divergence of $Z$ cannot identically
vanish inside the manifold $W$ bounded by $\gamma$ and $\gamma'$.
Therefore, there is no possibility of extending $Y|_{K}$ in a divergence-free
way to the whole $M$ so that the resulting vector field still coincides
with $X$ in $U^{c}$. At first glance, one may think that the above
obstruction might be overcome if one can find an alternative method
for the construction of the extension $Z$ of $Y|_{K}$ that renounces
to obtain $Z=X$ in $U^{c}$. 

Even so, the answer may still be negative. Indeed, the desired divergence-free
extension of $Y|_{K}$ might simply not exist at all (see Example
\hyperlink{Ex2}{2} below). Note that while by hypothesis, $\gamma$
is the boundary of \emph{no} $n$-submanifold contained in $U$, it
may still be the boundary of an $n$-submanifold $W$ \emph{not }contained
in $U$ (i.e. $\gamma$ may be null-cobordant in $M$). Now, by the
divergence theorem, the flux of the original vector field $X$ across
$\gamma$ is zero, but the divergence-free $C^{r}$-perturbation $Y$
of the restriction of $X$ to $U$ may change this flux to a non zero
value. But then, no $C^{1}$ extension of $Y|_{K}$ to the whole $M$
can have a divergence that identically vanishes inside $W$. 

These obstructions can be removed at once if we make a simple and
natural topological assumption, namely that $U\setminus K$ is connected.
This implies the existence of a compact $n$-submanifold $P$ with
smooth connected boundary such that $K\subset\text{int}\,P$ and $P\subset U$
(Lemma \hyperlink{l3}{3}), which is the key to the construction of
the pasting of $Y$ and $X$ by the procedure described above. This
pasting result (Theorem \hyperlink{th1}{1}), which can also be formulated
in the Hölder setting (Theorem \hyperlink{th3}{3}), is then briefly
the following:\vspace{0.2cm}

\noindent (Conservative $C^{r}$ Pasting Lemma). \emph{Let $M$ be
a closed connected manifold, $U\subsetneq M$ an open neighbourhood
of a compact set $K$ such that $U\setminus K$ is connected and $r\in\mathbb{Z}^{+}$.
In the conservative setting, given any $C^{r}$ vector field $X$
on $M$ and any $C^{r}$-perturbation $Y$ of the restriction of $X$
to $U$, there exists a $C^{r}$-perturbation $Z$ of $X$ that coincides
with $Y$ in a neighbourhood of $K$ and with $X$ in $U^{c}$. }\vspace{0.2cm}

Theorem \hyperlink{th1}{1} also shows that vector field $Z$ can
be obtained so that the $C^{r}$ norm of $Z-X$ is linearly bounded
by that of $Y-X|_{U}$, for some fixed constant $C>1$ depending only
on $r$, $K$ and $U$ (and, of course, on the manifold's atlas, which
is assumed to be fixed).

The proof is constructive, elementary and self contained. It\textbf{
}essentially relies on a simple but ingenious global-to-local reduction
procedure originally due to Moser \cite{MO}. Besides its simplicity,
the main advantage of Moser's direct approach is the guaranty that
the auxiliary divergence-canceling vector field $v$ satisfying $\mbox{div }v=\mbox{div }\widehat{Z}$
will be (compactly) supported inside the open ``transition annulus''
$\varOmega\subset U\setminus K$ (the set where the transition from
vector field $Y$ to $X$ is set to take place; in practice, it will
correspond to a small neighbourhood of the closure of $\{x\in M:\,0<\xi(x)<1\}$),
and thus extends by $0$ to the whole $M$ (in the $C^{r}$ class).
This is needed to guarantee that the divergence canceling operation
$\widehat{Z}-v$ producing $Z$ does not change $\widehat{Z}$ outside
$\varOmega$, so that $Z$ still coincides with $Y$ and $X$ in $K$
and $U^{c}$, respectively. Due to the linearity of the divergence
operator, the use of optimal regularity tools of Dacorogna-Moser type
\cite[Theorem 2]{DM} (which are crucial in the discrete-time case,
see Sections \hyperlink{se1.2}{1.2} and \hyperlink{se4}{4}) can
be entirely avoided, as there is no regularity loss in the divergence
of the initial (non-conservative) pasting perturbation: if $X$ and
$Y$ are divergence-free $C^{r}$ vector fields and $\widehat{Z}$
is a vector field defined as above, then $\mbox{div }\widehat{Z}$
is still of class $C^{r}$ and $C^{r}$ small if $Y$ is $C^{r}$
close to $X$.

This conservative pasting lemma permits to establish several perturbation
tools of which three illustrative examples are singled out:
\begin{enumerate}
\item Localized smoothing (Theorem \hyperlink{th5}{5}): at least for certain
useful open sets $\varOmega\subset M$ (see Footnote 3),\emph{ one
may conservatively $C^{r}$ perturb a divergence-free vector field
$X$ in order to make it smooth inside $\varOmega$, while keeping
$X$ unchanged on the complement of $\varOmega$}.
\item Perturbative extension with increased regularity (Corollary \hyperlink{C1}{1}):\emph{
if a $C^{r}$-perturbation $Y$ of the restriction of $X$ to $U$
is of class $C^{s}$, $s>r$, $C^{r}$ being the regularity of $X$,
then $Y|_{K}$ can be (conservatively) extended to a $C^{r}$-perturbation
of $X$ which is of class $C^{s}$ on the whole $M$.}
\item Local linearization of ``Franks lemma type'' (Theorem \hyperlink{th6}{6}):
\emph{one may conservatively $C^{1}$-perturb a vector field $v$
near a point $x$ (keeping $v(x)$ unaltered), in order to change
its derivative at $x$ and make $v$ affine linear near this point,
the allowed variation $\delta$ of the derivative depending linearly
on the required $C^{1}$-closeness $\epsilon$ of the resulting vector
field to $v$ }(this result requires the use of an additional homothety
trick).
\end{enumerate}
Other examples could be given, however the primary intention of this
work is to present a few solid basic techniques that might serve as
a starting point for the development of more sophisticated conservative
tools. Special care has been taken to ensure that:\smallskip{}

(a) The results obtained are the best possible both in terms of the
regularity of the resulting vector field or diffeomorphism as in terms
of the regularity of the closeness of the resulting system to the
original one. In the case of volume preserving diffeomorphisms (see
Sections \hyperlink{se1.2}{1.2} and \hyperlink{se4}{4}) this endeavour
is restricted by the limits of the present knowledge concerning the
existence of optimal regularity solutions to the prescribed Jacobian
PDE (which is an open problem in the $C^{r}$ case, $r\in\mathbb{Z}^{+}$
\cite[p.192]{CDK}). 

\smallskip{}

(b) The linear dependence of $\delta$ on $\epsilon$ is established
in all perturbation results (with the exception of Theorems \hyperlink{th4}{4}
and \hyperlink{th5}{5} where this is meaningless). Obtaining this
dependence is often crucial in applications.

\smallskip{}

(c) The proofs presented are constructive whenever possible and complete
or at least easily completable following the indications in the text.\smallskip{}

The pasting technique for divergence-free vector fields was introduced
by Arbieto and Matheus in \cite{AM}. It is known, however, that the
statements and proofs of the main tools (\cite[Section 3.1]{AM})
are not quite correct (concerning the statements, see Warnings \hyperlink{W1}{1}
and  \hyperlink{W2}{2} below). Some of the problems have been identified
in \cite{AM2}, but we are unaware of any reference correctly stating
this kind of results and providing sound proofs. The writing of the
present work was partially stimulated by the author's encounter with
that paper.

\hypertarget{se1.2}{}

\subsection{Discrete-time dynamics}

We now turn to the case of volume-preserving diffeomorphisms. To establish
in this setting a $C^{r}$-perturbation pasting lemma analogue to
Theorem \hyperlink{th1}{1} seems beyond the techniques presently
available (see (a) above and Section \hyperlink{se4}{4}), the main
difficulty being that volume correcting $C^{r}$ diffeomorphism (playing
the analogue role to the divergence-canceling vector field $v$ in
Section \hyperlink{se1.1}{1.1}) must now be reconstructed from a
determinant which is only of class $C^{r-1}$ and  $C^{r-1}$-close
to $1$. Nevertheless, using optimal regularity tools with control
of support, such result can actually be established in the Hölder
setting, but special care must be taken due to the pathological continuity
behaviour of the composition and inversion operators in these functional
spaces. This result will be presented in a separate note \cite{TE2}.
Here, we shall restrict to establish a quite general conservative
linearized version of Franks lemma, an important feature being that
the resulting diffeomorphism will have the same $C^{r,\alpha}$ regularity
as the original one $(r\in\mathbb{Z}^{+}$, $0<\alpha<1$). As it
is well known, to achieve the local affine linearization (and not
merely the perturbation of the derivative) is often essential to guarantee
the control of the dynamics near the perturbed fixed point or periodic
orbit, specially when the perturbed derivative is non-hyperbolic,
as it was already evident in the original paper \cite{FR}. Another
important aspect as far as applications are concerned is to establish
the linear dependence of the permitted variation $\delta$ of the
derivative in terms of the required $C^{1}$-closeness $\epsilon$
to the original diffeomorphism. As in \cite{FR}, this linear dependence
is also established in Theorem \hyperlink{th8}{8}. It is interesting
to compare the later result both with (a) the original Franks lemma
and with (b) the corresponding result for vector fields (Theorem \hyperlink{th6}{6}).
In all the three results, the resulting diffeomorphism or vector field
has the same regularity as the original one and the linear dependence
of $\delta$ on $\epsilon$ is established, but while (a) and (b)
are quite elementary, the Hölder case of Theorem \hyperlink{th8}{8}
requires the use of optimal regularity tools with control of support
and has much deeper roots, ultimately relying on the elliptical regularity
solutions to the Poisson problem with Neumann boundary condition and
the corresponding Schauder estimates (see \cite{DM}). The solution
in the $C^{\infty}$ case is simpler, relying on Moser's elegant yet
powerful flow method. In both cases, the starting point is a homothety
trick that proved crucial in establishing Avila's regularization \cite{AV}.
Note, however, that the results in \cite{DM} cannot be directly applied
in the present context, due to their lack of control of support (see
(ii) below). One uses instead their counterparts in \cite{TE} where
this control is achieved (the proofs of the later results follow closely
the original ones in \cite{DM}). In the dynamical systems literature,
Dacorogna-Moser's powerful theorems have been often misinterpreted
and naively applied in several ways. As these flaws are somewhat recurrent,
it is perhaps not out of place to call here attention to them:

(i) In \cite{DM} it is necessary to assume that the domain $\varOmega$
is \emph{connected }(besides bounded). This was omitted by lapse in
the statements of the propositions, but it is explicitly assumed at
the beginning of page 2.

(ii) In \cite[Theorem 1']{DM}, the solution diffeomorphism $\varphi$
in general does \emph{not} extend by the identity to the whole $\mathbb{R}^{n}$
in the $C^{k+1,\alpha}$ class, not even when the determinant $f$
equals 1 in a neighbourhood of $\partial\varOmega$. For instance,
in order to guarantee that a volume correcting diffeomorphism acts
only inside the region $\varOmega$ where the volume distortion takes
place (i.e. that $\text{supp}(\varphi-\text{Id})\subset\overline{\varOmega})$
one needs instead the corresponding results with control of support
as in \cite{TE}. Analogue observation holds for the linearized problem
$\text{div}\,u=h$.

(iii) The optimal regularity statements in \cite{DM} and \cite{TE}
with $C^{k,\alpha}$ replaced by $C^{k}$, $k\in\mathbb{Z}^{+}$,
have \emph{not }been established in dimension higher than one (being
false for $k=0$ \cite[p.192 and 180]{CDK}).

(iv) Concerning the regularity of the solution diffeomorphism $\varphi$
in \cite[Theorems 1']{DM} when the determinant $f$ is $C^{\infty}$
see part (B) in the proof of Lemma \hyperlink{l2}{2} below.

\hypertarget{se2}{}

\section{Conservative pasting, extension, localized smoothing and local linearization
of vector fields}

\noindent \textbf{Convention. }Throughout this paper, $M$ \emph{is
a (second countable, Hausdorff)} \emph{connected orientable closed}
$C^{\infty}$ \emph{manifold of dimension} $n\geq2$, \emph{equipped
with a finite atlas $(V_{i},\,\phi_{i})_{i\leq m}$ and a} $C^{\infty}$
\emph{volume form} $\omega$. By \cite{MO},\footnote{As remarked in \cite[p.4 and 23]{DM}, the proof given in \cite[Lemma 2]{MO}
is actually for that proposition with both the hypothesis $\text{supp(}g-h)\subset Q$
and the conclusion $\text{supp}(u-\text{Id})\subset Q$ removed ($g,\,h$
being the restrictions to the open $n$-cube $Q$ of two smooth volume
forms defined on $\overline{Q}$ and having the same total volume,
the proof produces a smooth diffeomorphism $u$ realizing a pullback
between them).} we can assume that the atlas is \emph{conservative}, i.e. on each
local chart, $\omega$ pushes forward to the canonical volume form
on $\mathbb{R}^{n}$ and $\phi_{i}(V_{i})=\lambda\mathbb{B}^{n}$,
for some constant $\lambda>0$; $\mu$ is the Lebesgue measure induced
by $\omega$ on $M$. We may further assume that the atlas is \emph{regular
}in the sense that there is a ``larger'' conservative atlas $(W_{i},\varPhi_{i})_{i\leq m}$
such that $\overline{V_{i}}\subset W_{i}$ and $\varPhi_{i}|_{V_{i}}=\phi_{i}$.
As usual, $\mathbb{B}^{n}$ is the (open unit) $n$-ball in Euclidean
space and $\mathbb{D}^{n}=\overline{\mathbb{B}^{n}}$ is the $n$-disk.
\medskip{}

Given an open set $U\subset M$, denote by $\mathfrak{X}^{s}(U)$,
$s\in\mathbb{Z}^{+}\cup\{\infty\}$, the space of\emph{ }vector fields
of class $C^{s}$ defined on \emph{$U$} and by $\mathfrak{X}_{\mu}^{s}(U)$
the subspace of those that are divergence-free in relation to $\omega$,
or equivalently, whose flows preserve $\mu$. As mentioned in the
Introduction, in Theorem \hyperlink{th1}{1} we consider vector fields
$Y$ defined on open sets $U\subset M$, which are $C^{r}$-perturbations
of $X|_{U}$, $X$ being a vector field in $\mathfrak{X}_{\mu}^{s}(M)$.
To guarantee that the $C^{r}$ norms of these $Y$ remain finite,
we introduce the following

\hypertarget{def1}{}
\begin{defn}
(\emph{$C^{r}$}-bounded) Let $r,\,s\in\mathbb{Z}^{+}\cup\{\infty\}$,
$r\leq s$. $Y\in\mathfrak{X}^{s}(U)$ is $C^{r}$\emph{-bounded}
if $Y$ and all its derivatives up to order $r$ are bounded on $U$.
$\left\Vert \cdot\right\Vert _{C^{r};U}$ is Whitney $C^{r}$ norm
($\mathbb{N}_{0}\ni r\leq s)$ on $\mathfrak{X}^{s}(U)$ (Section
\hyperlink{se5.1}{5.1}). When $U=M$ we simply write $\left\Vert \cdot\right\Vert _{C^{r}}$.

\medskip{}
\end{defn}
We recall the informal description of Theorem \hyperlink{th1}{1}.
\emph{In the volume preserving setting}, let $X$ be a vector field
of class $C^{r}$ on a closed manifold $M$ and $U\subsetneq M$ an
open neighbourhood of a compact set $K$. Given a $C^{r}$ perturbation
$Y$ of the restriction of $X$ to $U$, it is possible (provided
$U\setminus K$ is connected), to $C^{r}$-perturb $X$ inside $U$
only, so that the resulting vector field on $M$ still coincides with
$Y$ in some neighbourhood of $K$. One interesting point is that
the perturbation can be made $C^{\infty}$ in the open set where the
control over the dynamics is necessarily lost, i.e. on the ``transition
annulus'' where the conservative ``harmonization'' of the two original
vector fields takes place (this being the unavoidable cost of bringing
together in the same vector field two more or less ``conflicting''
dynamics). 

\hypertarget{th1}{}
\begin{thm}
\emph{($C^{s}$ conservative pasting with $C^{r}$-closeness)}. Let
$M$ be a manifold as above. Suppose that $K$ is a compact subset
with an open neighbourhood $U\subsetneq M$ such that $U\setminus K$
is connected. Then, given \textup{$s\in\mathbb{Z}^{+}\cup\{\infty\}$}
and an integer $1\leq r\leq s$, there is an open set $K\subset V\subset U$
and a constant $C=C(r,K,U)>1$ such that: given $X\in\mathfrak{X}_{\mu}^{s}(M)$
and a $C^{r}$-bounded $Y\in\mathfrak{X}_{\mu}^{s}(U)$, there exists
$Z\in\mathfrak{X}_{\mu}^{s}(M)$ satisfying:

\begin{enumerate}
\item \hypertarget{1}{}$Z=Y$ in $V$;
\item \hypertarget{2}{}$Z=X$ in a neighbourhood of $U^{c}$;
\item \hypertarget{3}{}$\left\Vert Z-X\right\Vert _{C^{r}}\leq C\left\Vert Y-X\right\Vert _{C^{r};U}$
\end{enumerate}
\noindent Moreover, $V$ depends only on $K$ and $U$ and not on
$r,\,s$ and one may further require $Z$ to be $C^{\infty}$ at every
point where it neither coincides with $X$ nor with $Y$.

\end{thm}
Actually the proof establishes a considerably more precise result
(as usual, $Y\not\equiv X|_{U}$ means that $Y(x)\neq X(x)$ for some
point $x\in U$): 

\hypertarget{th2}{}
\begin{thm}
\noindent Let M, $K$, $U$, $r$ and $s$ be as above. Then, there
is a constant $C=C(r,K,U)>1$ and two disjoint compact $n$-submanifolds
$Q$ and $S$ with smoothly diffeomorphic connected boundaries for
which $K\subset\text{\emph{int}}\,Q$ and $U^{c}\subset\text{\emph{int}\,}S$
and such that: given $X\in\mathfrak{X}_{\mu}^{s}(M)$ and a $C^{r}$-bounded
$Y\in\mathfrak{X}_{\mu}^{s}(U)$ such that $Y\not\equiv X|_{U}$,
there exists $Z\in\mathfrak{X}_{\mu}^{s}(M)$ satisfying:
\begin{enumerate}
\item $Z=Y$ in $Q$;
\item $Z=X$ in $S$;
\item $Z$ is $C^{\infty}$ in $\varOmega=(Q\cup S)^{c},$ $\varOmega$
being $C^{\infty}$ diffeomorphic to $\partial{\color{blue}Q}\times]0,1[$;
\item $\left\Vert Z-X\right\Vert _{C^{r}}\leq C\left\Vert Y-X\right\Vert _{C^{r};U}$
\end{enumerate}
\noindent Moreover, $Q$ and $S$ depend only on $K$ and $U$ and
not on $r,\,s$.
\end{thm}
(Note that if $Y\equiv X|_{U}$, then inequality (4) implies that
$Z\equiv X$ on the whole $M$, thus one cannot, in general, guarantee
the conclusion (3) in this case).

\hypertarget{r1}{}
\begin{rem}
(Hölder setting). Theorem \hyperlink{th1}{1} is still valid in the
Hölder setting (i.e for divergence-free vector fields of class $C^{s,\beta}$
endowed with a possibly lower $C^{r,\alpha}$ norm), the unique exception
being that one may require $Z$ to be $C^{\infty}$ in the set of
points where $Z$ neither coincides with $X$ nor with $Y$ essentially
only when $r+\alpha<s+\beta$ (smooth maps being in general only $C^{r,\rho}$-dense
in the class of $C^{r,\alpha}$ maps, $0<\rho<{\color{blue}{\normalcolor \alpha}\leq1}$;
see Theorem \hyperlink{th3}{3} below for the notation). We observe
that while the previous density remark implies that the analogue of
conclusion (3) in Theorem \hyperlink{th2}{2} is impossible to obtain
when $r=s$ and $0<\alpha=\beta\leq1$,\footnote{The case $r+\alpha=s+\beta$ splits into 3 subcases: (a) the one just
mentioned; (b) $r=s$, $\alpha=\beta=0$, which is the $C^{r}$ case
(Theorem \hyperlink{th1}{1}); (c) $s=r+1$, $\alpha=1$, $\beta=0$,
which again reduces to Theorem \hyperlink{th1}{1}, the norms $C^{r,1}$
and $C^{r+1}$ being equivalent (\cite[p.342]{CDK}). Thus, only (a)
is ``Hölder relevant''.} the remaining relevant Hölder case $r+\alpha<s+\beta$ is actually
free from these constraints. In particular, using the above mentioned
(Euclidean space) Hölder density result in place of the $C^{r}$-density
of $C^{\infty}$ in $C^{r}$, the proof of Theorem \hyperlink{th4}{4}
immediately yields that $\mathfrak{X}_{\mu}^{\infty}(M)$ is $C^{r,\alpha}$-dense
in $\mathfrak{X}_{\mu}^{s,\beta}(M)$, when $r+\alpha<s+\beta$. Note,
however, that, a priori, this is not enough to obtain the corresponding
Hölder version of Theorem \hyperlink{th5}{5} (which, by its turn,
is used to obtain conclusion (3) in Theorem \hyperlink{th2}{2} above),
as the resulting vector field $Z$ would still be obtained as the
limit of a sequence of $C^{s,\beta}$ vector fields, which sequence
is Cauchy only in relation to the lower $C^{r,\alpha}$-norm, and
this is not enough to ensure that $Z$ belongs to the higher class
$C^{s,\beta}$ as required. Nevertheless, this problem can be overcome
by a simple lower semicontinuity reasoning: using \cite[(7.14), p.148 and Lemma 7.3, p.150]{GT}
one sees that modifying the proof of Theorem \hyperlink{th4}{4} as
explained above, the sequence $Z_{k}$ of smooth, divergence-free
vector fields $C^{r,\alpha}$-converging to $X$ has $C^{s,\beta}$
norm uniformly bounded by that of $X$ times a constant. Now, carrying
the proof of Theorem \hyperlink{th5}{5} using these smooth approximations
to $X$, it is immediate to check that an analogue uniform boundeness
of the $C^{s,\beta}$ norms also holds for all the auxiliary functions
and vector fields involved in the construction of the Cauchy sequence
$Z_{k}$ (the universality of the operator $\varPhi$ in Lemma \hyperlink{l1}{1}
being essential here). This finally yields that the $C^{s,\beta}$
norms of the vector fields in this sequence are still uniformly bounded
by the $C^{s,\beta}$ norm of $X$ times a constant (which is independent
of $X$). This guarantees that the limit vector field $Z$ actually
belongs to the $C^{s,\beta}$ class by lower semicontinuity (see e.g.
\cite[p.358]{CDK}). We finally observe that the existence of manifolds
$Q$ and $S$ satisfying (1) - (3) as in Theorem \hyperlink{th2}{2}
also holds for Theorem \hyperlink{th3}{3}, except that (as explained
above) one cannot guarantee $Z$ to be $C^{\infty}$ in $\varOmega$
when $r=s$ and $0<\alpha=\beta\leq1$.

In Section \hyperlink{se3.2}{3.2} we briefly outline the few changes
needed in the proof of Theorem \hyperlink{th1}{1} to obtain Theorem
\hyperlink{th3}{3}. There, it is also explained why constant $C$
actually does not depend on the Hölder exponent $\alpha$, but only
on $r$, $K$ and $U$.

Given an open set $U\subset M$, $s\in\mathbb{Z}^{+}$ and $0<\beta\leq1$,
$\mathfrak{X}^{s,\beta}(U)$ is the subspace of $\mathfrak{X}^{s}(U)$
consisting of vector fields $Y$ such that, on local charts, each
partial derivative of $Y$ of order $s$ is $\beta$-Hölder continuous
(these derivatives being functions from $\phi_{j}(V_{j}\cap U)$ into
$\mathbb{R}^{n}$). One sets $C^{s,0}:=C^{s}$ and $C^{\infty,\beta}:=C^{\infty}$.
\end{rem}
\hypertarget{th3}{}
\begin{thm}
\emph{($C^{s,\beta}$ conservative pasting with $C^{r,\alpha}$-closeness)}.
Let $M$ be a manifold as above. Suppose that $K$ is a compact subset
with an open neighbourhood $U\subsetneq M$ such that $U\setminus K$
is connected. Then, given \textup{$s\in\mathbb{Z}^{+}\cup\{\infty\}$,
}$0\leq\alpha,\beta\leq1$\textup{,} and an integer $1\leq r\leq s$
such that $r+\alpha\leq s+\beta$, there is an open set $K\subset V\subset U$
and a constant $C=C(r,K,U)>1$ such that: given $X\in\mathfrak{X}_{\mu}^{s,\beta}(M)$
and a $C^{r}$-bounded $Y\in\mathfrak{X}_{\mu}^{s,\beta}(U)$, there
exists $Z\in\mathfrak{X}_{\mu}^{s,\beta}(M)$ satisfying:

\begin{enumerate}
\item $Z=Y$ in $V$;
\item $Z=X$ in a neighbourhood of $U^{c}$;
\item $\left\Vert Z-X\right\Vert _{C^{r,\alpha}}\leq C\left\Vert Y-X\right\Vert _{C^{r,\alpha};U}.$
\end{enumerate}
\noindent Moreover, $V$ depends only on $K$ and $U$ and with the
exception of the case $r=s$ and $0<\alpha=\beta\leq1$, one may further
require $Z$ to be $C^{\infty}$ at every point where it neither coincides
with $X$ nor with $Y$.
\end{thm}
\medskip{}

\hypertarget{W1}{}

\noindent \emph{Warning }1\emph{.} It should be stressed that if $U\setminus K$
is \emph{not} connected, then cobordism constraints might occur making
(in general) impossible the conservative pasting of vector fields
$X$ and $Y$ as stated in Theorems \hyperlink{th1}{1}, \hyperlink{th2}{2}
and \hyperlink{th3}{3} (by \emph{any }method and under \emph{any}
regularity assumptions, see Example \hyperlink{Ex1}{1}). If $U\setminus K$
fails to be connected, a conservative $C^{r}$ perturbation $Y$ of
the restriction of $X$ to $U$ may actually fail to have a divergence-free
extension to the whole $M$, even if the $C^{r}$ closeness condition
is dropped (Example \hyperlink{Ex2}{2}).

\hypertarget{Ex1}{}
\begin{example}
Represent the flat 2-torus as $M=\mathbb{S}^{1}\times(\mathbb{R}/\mathbb{Z})$
with coordinates $(s,z)$ and endow it with the standard volume form.
Let $X$ be the vertical vector field $\frac{\partial}{\partial z}$
and consider its $\epsilon\text{-}$$C^{\infty}$ perturbation $Y=(1+\epsilon$)$\frac{\partial}{\partial z}$,
$\epsilon>0$. Then, there is no $Z\in\mathfrak{X}_{\mu}^{1}(M)$
such that (a) $Z=Y$ in $K=\mathbb{S}^{1}\times1/2$ and (b) $Z=X$
in $\gamma=\mathbb{S}^{1}\times1$. The vector fields $X$ and $Y$
have a different flux across the cobordant circles $\mathbb{S}^{1}\times z$,
thus the divergence of $Z$ cannot identically vanish inside any of
the two annulus bounded by $K$ and $\gamma$.
\end{example}
\hypertarget{Ex2}{}
\begin{example}
Extend the annulus $U=\mathbb{S}^{1}\times]-1,1[\,\subset\mathbb{R}^{3}$
to a smoothly embedded 2-sphere $S$, invariant under rotation about
the $z$-axis and endowed with the canonical volume form inherited
from $\mathbb{\mathbb{R}}^{3}$. Endow $S$ with the rotation vector
field $X:(x,y,z)\mapsto(-y,x,0)$ and consider the $\epsilon\text{-}$$C^{\infty}$
perturbation of the restriction of $X$ to $U$ given by $Y=(-y,x,\epsilon)$,
$\epsilon>0$. Both $X$ and $Y$ are divergence-free but there is
no $C^{1}$ divergence-free extension of $Y$ to the whole $S$, as
the flux of $Y$ across the boundary circle $\gamma=\mathbb{S}^{1}\times0$
is not zero, the divergence being necessarily positive around some
point of the southern hemisphere.

\medskip{}
\end{example}
\hypertarget{W2}{}

\noindent \emph{Warning }2\emph{.} The dependence of constant $C$
on $r$, $K$ and $U$ is obviously unavoidable, whatever the method
employed to achieve the pasting of the vector fields. For instance,
given a point $p\in M$, in some local chart set $K=\{p\}$ and $U=B_{d}(p)$
a small open ball whose closure is contained in the chart. Since $d=\mbox{dist}(K,\,U^{c})$,
the mean value theorem then implies that $C=C(r,K,U)>d^{-r}$. In
general and by the same reason, for $K$ and $U$ as in Theorem \hyperlink{th1}{1},
a ``thin'' $U\setminus K$ implies a quite large $C$. More precisely,
assume for the moment that $M$ is endowed with a Riemannian structure
inducing an intrinsic metric (this structure is actually unnecessary
for the results here obtained). Suppose that for each $\epsilon>0$,
$U_{\epsilon}$ is an open neighbourhood of $K$ contained in $B_{\epsilon}(K)$
with $U_{\epsilon}\setminus K$ connected. Then $C(r,K,U_{\epsilon})\rightarrow\infty$
as $\epsilon\rightarrow0$. To get an idea of how the ``geometry''
of $U\setminus K$ tends to impact the size of $C$, and in the specific
context of the method employed here to solve equation $\mbox{div}\,v=h$,
observe that, roughly speaking, the ``thinner'' and possibly more
``convoluted'' the image of $U_{\epsilon}\setminus K$ on the atlas
as $\epsilon$ tends to zero, the larger the number $N+1$ of small
cubes $U_{j}$ needed to achieve the covering 
\[
\overline{\varOmega_{1}}\subset\bigcup_{j=0}^{N}U_{j}\subset\varOmega\subset{\normalcolor {\color{blue}{\normalcolor U_{\epsilon}}}}\setminus K
\]
(see the proof of Theorem \hyperlink{th1}{1}) and a large number
of small cubes contributes to $C$ with a very large multiplicative
factor (the smaller the cubes the larger this factor becomes, see,
in particular, Section \hyperlink{3c}{3.(c)} and Footnote 5). Together,
this and the previous Warning impose double caution on the use of
the pasting lemma to attempt general perturbations of divergence-free
vector fields with a priori unspecified support (however, see Theorem
\hyperlink{th6}{6}). As stated, with $\delta$ independent of $K$
and $U$, Theorem 3.1 in \cite{AM} contradicts the mean value theorem,
assuming, as implicit, that $W^{c}$ is nonempty (in the paper's notation). 

\medskip{}

\hypertarget{se2.1}{}

\subsection{Conservative localized smoothing and extension.}

The proof of next result corrects and generalizes that of \cite[Theorem 2.2]{AM}.
It provides a short alternative way to establish Zuppa's regularization
theorem \cite{ZU} without the need to introduce a Riemannian structure
on the manifold.

\hypertarget{th4}{}
\begin{thm}
Let $M$ be a manifold as above and $r\in\mathbb{Z}^{+}$. Then, 
\[
\mathfrak{X}_{\mu}^{\infty}(M)\text{ \,is \,}C^{r}\text{-dense in \,}\mathfrak{X}_{\mu}^{r}(M).
\]
\end{thm}
\begin{proof}
Let $(V_{i},\phi_{i})_{i\leq m}$ be the atlas of $M$. There is no
difficulty in finding a partition of unity $\xi_{i\leq m}$ subordinate
to $V_{i\leq m}$ with $\xi_{1}=1$ in $\phi_{1}^{-1}(\frac{2\lambda}{3}\mathbb{D}^{n})$
(see the Convention above). Let $X_{i}=(X_{i}^{1},\ldots,X_{i}^{n})$,
$i\leq m$, be the expressions of $X\in\mathfrak{X}_{\mu}^{r}(M)$
in the local charts. Since the atlas is regular (see the Convention,
Section \hyperlink{se2}{2}), using convolutions one can find, for
each $i$, a sequence $X_{ik}$ of smooth vector fields on $\phi_{i}(V_{i})=\lambda\mathbb{B}^{n}\,$
$C^{r}$-converging to $X_{i}$. Observe that, as $X_{i}$, each $X_{ik}$
is divergence-free (in relation to the standard volume form on $\mathbb{R}^{n}$),
since the convolution operator $*$ is bilinear and satisfies $\partial_{j}(\rho*X_{i}^{j})=\rho*\partial_{j}X_{i}^{j}$.
To simplify the notation, one still denotes by $X_{ik}$ the pullback
$\phi_{i}^{*}(X_{ik})$. Define the smooth vector field on $M$,
\[
Y_{k}=\sum_{i\leq m}\xi_{i}X_{ik}
\]
setting $\xi_{i}X_{ik}:=0$ in $V_{i}^{c}$. Since $\sum_{i\leq m}\xi_{i}=1$,
the estimate for the $\left|\cdot\right|_{r}$ norm of the product
(end of Section \hyperlink{se5.1}{5.1}) gives
\begin{equation}
\left|Y_{k}-X\right|_{r}=\Bigl|\sum_{i\leq m}\xi_{i}(X_{ik}-X)\Bigr|_{r}\leq m2^{r}\underset{i\leq m}{\text{max}}\left|\xi_{i}\right|_{r}\,\underset{i\leq m}{\text{max}}\big|X_{ik}-X|_{V_{i}}\big|_{r}
\end{equation}
Since $X$ and the $X_{ik}$'s are divergence-free in $M$ and $V_{i}$,
respectively, and $\xi_{i}$ is compactly supported inside $V_{i}$,
\[
\text{div\,}Y_{k}=\text{div\,}Y_{k}-\text{div}\,X=\text{div}\,(Y_{k}-X)=\sum_{{\scriptscriptstyle i\leq m;\,j\leq n}}(\partial_{j}\xi_{i})\bigl(X_{ik}^{j}-X^{j}\bigr)
\]
and 
\begin{equation}
\left|\text{div\,}Y_{k}\right|_{r}\leq mn2^{r}\underset{i\leq m}{\text{max}}\bigl|\xi_{i}\bigr|_{r+1}\,\underset{i\leq m}{\text{max}}\big|X_{ik}-X|_{V_{i}}\big|_{r}
\end{equation}
Since the norms $\left|\cdot\right|_{r}$ and $\left\Vert \cdot\right\Vert {}_{C^{r}}$
are equivalent (Section \hyperlink{se5.1}{5.1}) we work with the
former. From (2.1) and (2.2) it follows that
\begin{equation}
\text{ \ensuremath{\left|Y_{k}-X\right|_{r},\,\left|\text{div\,}Y_{k}\right|_{r}\xrightarrow[k\rightarrow\infty]{}0}\,\, since\, \,}X_{ik}\xrightarrow[k\rightarrow\infty]{C^{r}}X|_{V_{i}}
\end{equation}
Now, $\text{div\,}Y_{k}=0$ in $\mathcal{D}=\phi_{1}^{-1}(\frac{2\lambda}{3}\mathbb{D}^{n})$,
since, $\xi_{1}|_{\mathcal{D}}=1$ and thus $Y_{k}=X_{1k}$ in this
set. Let $\varOmega=M\setminus\phi_{1}^{-1}(\frac{\lambda}{3}\mathbb{D}^{n})$
and $\varOmega_{1}=M\setminus\phi_{1}^{-1}(\frac{\lambda}{2}\mathbb{D}^{n})$.
Let $h_{k}=\text{div}\,Y_{k}$. Clearly $\overline{\varOmega_{1}}\subset\varOmega$
and $\text{supp}\,h_{k}\subset\varOmega_{1}$. Observe that $\varOmega$
is connected since $M$ and $\partial\varOmega$ (diffeomorphic to
$\mathbb{S}^{n-1}$) are both connected, and the same holds for $\varOmega_{1}$.
Moreover, $\int_{\varOmega}h_{k}\omega=0$ ($\omega$ being the volume
form on $M$), since by the divergence theorem,
\[
\int_{\varOmega}\text{(div}\,Y_{k})\omega=\int_{\partial\varOmega}Y_{k}\lrcorner\,\omega=-\int_{\partial\mathcal{B}}X_{1k}\lrcorner\,\omega=-\int_{\mathcal{B}}\text{(div}\,X_{1k})\omega=0
\]
where $\mathcal{B}=\phi_{1}^{-1}(\frac{\lambda_{1}}{3}\mathbb{B}^{n})$.
Now, by Lemma \hyperlink{l1}{1} (below), there is a constant $C=C(r,\varOmega_{1},\varOmega)>0$
and $v_{k}\in\mathfrak{X}^{\infty}(M)$ such that
\begin{equation}
\begin{cases}
\text{div}\,v_{k}=h_{k}\\
\text{supp}\,v_{k}\subset\varOmega\\
\left|v_{k}\right|_{r}\leq C\left|h_{k}\right|_{r}
\end{cases}
\end{equation}
Let $Z_{k}=Y_{k}-v_{k}$. Then, $Z_{k}\in\mathfrak{X}_{\mu}^{\infty}(M)$
and finally by (2.3) and (2.4), 
\[
\left|Z_{k}-X\right|_{r}\leq\left|Y_{k}-X\right|_{r}+\left|v_{k}\right|_{r}\xrightarrow[k\rightarrow\infty]{}0
\]
\end{proof}
At least for certain open sets $\varOmega\subset M$,\footnote{In the preprint arXiv:1611.01694v3 to this paper, it was stated without
proof (unnumbered theorem on page 8) that Theorem \hyperlink{th5}{5}
below still holds for arbitrary open sets $\varOmega\subset M$. It
turned out that the proof known to the author contained an error.
Therefore, and to the best of our knowledge, the general case remains,
so far, conjectural.} which turn out to be useful in many important situations, one may
conservatively $C^{r}$-perturb a divergence-free vector field $X$
in order to make it smooth inside $\varOmega$, while keeping $X$
unchanged on the complement of $\varOmega$. This result has the advantage
of avoiding the occurrence of a ``transition annulus'', where typically
$Z$ is neither smooth nor it coincides with $X.$ If, for instance,
one needs to perform a preliminary conservative $C^{r}$ perturbation
of a vector field $X$ in order to increase its regularity, it may
be actually possible to smooth it just where this is really needed
for the construction of the subsequent perturbations (e.g. on small
open neighbourhoods of certain periodic orbits in dimension $n\geq3$),
while keeping $X$ unchanged on the complement of that set. The advantages
in terms of dynamical control are evident.\medskip{}

Given a compact $n$-submanifold $N\subset M$ ($n=\text{dim\,}M$)
with $C^{r\geq2}$ boundary, one may construct a $C^{r-1}$ vector
field transverse to $\partial N$ and pointing inward, which by its
turn defines a $C^{r-1}$ collar embedding $\zeta:\partial N\times[0,\infty[\hookrightarrow N$,
$\zeta(x,0)=x$. For each $\epsilon>0$, $\zeta(\partial N\times[0,\epsilon])$
is a (compact $C^{r-1}$) \emph{collar} \emph{of }$\partial N$.

\hypertarget{th5}{}
\begin{thm}
\emph{(Conservative localized smoothing - special case).} Let $M$
be a manifold as above and $N\subset M$ a compact $n$-submanifold
with connected $C^{3}$ boundary. Let $\varOmega$ be either the interior
of $N$ or the interior of a (compact $C^{2}$) collar of $\partial N$.
Given $X\in\mathfrak{X}_{\mu}^{r}(M)$, $r\in\mathbb{Z}^{+}$, there
exists $Z\in\mathfrak{X}_{\mu}^{r}(M)$, as $C^{r}$-close to $X$
as desired, satisfying:
\begin{enumerate}
\item $Z$ is $C^{\infty}$ in $\varOmega$;
\item $Z=X$ in $\varOmega^{c}$.
\end{enumerate}
\end{thm}
\begin{proof}
(Case $\varOmega=\text{int}\,N$). Since the norms $\left|\cdot\right|_{r}$
and $\left\Vert \cdot\right\Vert _{C^{r}}$ are equivalent we work
with the former. Fix a $C^{2}$ collar embedding
\[
\zeta:\partial N\times[0,\infty[\hookrightarrow N
\]
Consider the open covering of $\varOmega=\text{int}\,N$ by overlapping
``annuli'' given by 
\[
\begin{cases}
\varLambda_{0}=\varOmega\setminus\zeta\bigl(\partial N\times]0,\frac{1}{3}]\bigr)\\
\varLambda_{k}=\zeta\bigl(\partial N\times]\frac{1}{2k+3},\frac{1}{2k}[\bigr), & k\geq1
\end{cases}
\]
and fix a smooth partition of unity $\xi_{k\geq0}$ of $\varOmega$
subordinate to it ($\varLambda_{0}$ is actually a $C^{2}$-isotopic
copy of $\varOmega$). Let 
\[
\varOmega_{k}=\varLambda_{k}\cap\varLambda_{k+1}
\]
Note that $\xi_{k}+\xi_{k+1}=1$ in $\varOmega_{k}$ by subordination
to the covering (we suggest to the reader the drawing of a figure).
Given $X\in\mathfrak{X}_{\mu}^{r}(M)$ and $\epsilon>0$ we shall
construct a sequence $Z_{k\geq0}\in\mathfrak{X}_{\mu}^{r}(M)$ such
that, for $k\geq0$
\begin{enumerate}
\item $Z_{k+1}=Z_{k}$ in $\varLambda_{k+1}^{c};$
\item $Z_{k}$ is $C^{\infty}$ in $(\varLambda_{0}\cup\cdots\cup\varLambda_{k})\setminus\varOmega_{k}$;
\item $Z_{k}=X$ in $(\varLambda_{0}\cup\cdots\cup\varLambda_{k})^{c}$;
\item $\left|Z_{0}-X\right|_{r}<{\textstyle \epsilon/2}$ and $\left|Z_{k+1}-Z_{k}\right|_{r}<\epsilon/2^{k+2}$.
\end{enumerate}
It follows that $Z_{k}$ is a Cauchy sequence converging to $Z\in\mathfrak{X}_{\mu}^{r}(M)$
in the Banach space $\mathfrak{X}_{\mu}^{r}(M)$, satisfying
\begin{itemize}
\item $Z$ is $C^{\infty}$ in $\varOmega=\overset{\infty}{\cup}\varLambda_{k}$;
\item $Z=X$ in $\varOmega^{c}=\overset{\infty}{\cap}(\varLambda_{0}\cup\cdots\cup\varLambda_{k})^{c}$;
\item $\left|Z-X\right|_{r}<\epsilon$.
\end{itemize}
(\textbf{A}. Construction of $Z_{0}$). Let $\widehat{Z_{0}}=\widehat{\xi}_{1}X+\xi_{0}X_{0}\in\mathfrak{X}^{r}(M)$
where $\widehat{\xi}_{1}=\xi_{1}$ in $\varLambda_{0}$ and $\widehat{\xi}_{1}=1$
elsewhere. Here $X_{0}\in\mathfrak{X}_{\mu}^{\infty}(M)$ is a vector
field whose $C^{r}$-closeness to $X$ will be determined below. Note
that $\widehat{Z_{0}}$ is divergence-free in $\varOmega_{0}^{c}$.
Actually, by subordination of the partition to the covering, there
is an open set 
\[
\varOmega_{0}^{*}=\zeta\bigl(\partial N\times]{\textstyle \frac{1}{3}}+\delta_{0},{\textstyle \frac{1}{2}-\delta}_{0}[\bigr)
\]
where $0<\delta_{0}<1/12$, such that $\overline{\varOmega_{0}^{*}}\subset\varOmega_{0}$
and $\text{supp}\,h_{0}\subset\varOmega_{0}^{*}$ for $h_{0}:=\text{div}\widehat{Z_{0}}$.
Observe that $\varOmega_{0}$ and $\varOmega_{0}^{*}$ are connected
($\partial N$ being connected) with $C^{2}$ boundary, thus by the
divergence theorem \cite[p.203]{LA},

\[
\begin{array}{lll}
\int_{\varOmega_{0}}h_{0}\omega=\int_{\partial\varOmega_{0}}\widehat{Z_{0}}\lrcorner\,\omega & =\underset{\,}{-\int_{\partial N_{0}}X_{0}\lrcorner\,\omega+\int_{\partial N_{0}^{*}}X\lrcorner\,\omega}\\
 & =\underset{\,}{-\int_{\text{int}\,N_{0}}\text{(div}\,X_{0})\omega+\int_{\text{int}\,N_{0}^{*}}\text{(div}\,X)\omega}\\
 & =-0+0=0
\end{array}
\]
since $\partial\varOmega_{0}=\partial N_{0}\sqcup\partial N_{0}^{*}$,
where for $k\geq0$,
\[
N_{k}=\varOmega\setminus\zeta\bigl(\partial N\times]0,{\textstyle \frac{1}{2k+2}}[\bigr)\text{ \,\,and \,\,}N_{k}^{*}=\varOmega\setminus\zeta\bigl(\partial N\times]0,{\textstyle \frac{1}{2k+3}}[\bigr)
\]
are manifolds $C^{2}$-isotopic to $N$. By Lemma \hyperlink{l1}{1}
(below), there is a constant $C=C(r,\varOmega_{0}^{*},\varOmega_{0})>0$
and $v_{0}=\varPhi(h_{0})\in\mathfrak{X}^{r}(M)$ such that
\[
\begin{cases}
\text{div}\,v_{0}=h_{0}\\
\text{supp}\,v_{0}\subset\varOmega_{0}\\
\left|v_{0}\right|_{r}\leq C\left|h_{0}\right|_{r}
\end{cases}
\]
Then, 
\[
Z_{0}=\widehat{Z_{0}}-v_{0}=\widehat{Z_{0}}-\varPhi(\text{div\,}\widehat{Z_{0}})\in\mathfrak{X}_{\mu}^{r}(M)
\]
is $C^{\infty}$ in $\varLambda_{0}\setminus\varOmega_{0}$ and $Z_{0}=X$
in $\varLambda_{0}^{c}$. Moreover, it is easily seen that if $|X_{0}-X|_{r}$
is small then $|\widehat{Z_{0}}-X|_{r}$, $|h_{0}|_{r}$ and consequently
$|v_{0}|_{r}$ are all small (see Section \hyperlink{se2.3}{2.3}
below), hence for $X_{0}$ sufficiently $C^{r}$-close to $X$,
\[
\left|Z_{0}-X\right|_{r}=\big|\widehat{Z_{0}}-v_{0}-X\big|_{r}\leq\big|\widehat{Z_{0}}-X\big|_{r}+\left|v_{0}\right|_{r}<{\textstyle \epsilon/2}
\]
(\textbf{B}. Construction of $Z_{1}$). Let
\[
\widehat{Z_{1}}=\widehat{\xi}_{2}X+\xi_{1}X_{1}+\xi_{0}X_{0}\in\mathfrak{X}^{r}(M)
\]
where $\widehat{\xi}_{2}=\xi_{2}$ in $\varLambda_{0}\cup\varLambda_{1}$
and $\widehat{\xi}_{2}=1$ elsewhere. Again, $X_{1}\in\mathfrak{X}_{\mu}^{\infty}(M)$
is a vector field whose $C^{r}$-closeness to $X$ is to be specified.
Now,\smallskip{}

(a)~~$\widehat{Z_{1}}$ is divergence-free in $(\varOmega_{0}\cup\varOmega_{1})^{c}$;\smallskip{}

(b)~~$\widehat{Z_{1}}=Z_{0}$ in $\varLambda_{1}^{c}$;\smallskip{}

(c)~~$\widehat{Z_{1}}$ is $C^{\infty}$ in $(\varLambda_{0}\cup\varLambda_{1})\setminus\varOmega_{1}$;\smallskip{}

(d)~~$\widehat{Z_{1}}=X$ in $(\varLambda_{0}\cup\varLambda_{1})^{c}$.\smallskip{}

Using Lemma \hyperlink{l1}{1}, we proceed exactly as in (A) to eliminate
the divergence of $\widehat{Z_{1}}$ inside $\varOmega_{0}$ and $\varOmega_{1}$,
while keeping this vector field unchanged in $(\varOmega_{0}\cup\varOmega_{1})^{c}$,
thus obtaining $Z_{1}\in\mathfrak{X}_{\mu}^{r}(M)$ as $C^{r}$-close
to $Z_{0}$ as desired and still satisfying (b) - (d) above (to establish
$\int_{\varOmega_{1}}h_{1}\omega=0$ where $h_{1}=\text{div}$$\widehat{Z_{1}}|_{\varOmega_{1}}$,
we now use $\partial\varOmega_{1}=\partial N_{1}\sqcup\partial N_{1}^{*}$
in order to apply the divergence theorem). 

As $\widehat{Z_{1}}=Z_{0}$ in $\varLambda_{1}^{c}$ and $\widehat{Z_{1}}=X_{1}$
and $Z_{0}=X$ in $\varLambda_{1}\setminus(\varOmega_{0}\cup\varOmega_{1})$,
and since we can take $X_{1}$ as $C^{r}$-close to $X$ as desired,
we need only to guarantee that $Z_{1}$ is as $C^{r}$-close to $Z_{0}$
as wished in $\varOmega_{0}\cup\varOmega_{1}$. With $\varOmega_{1}$
there is no concern, the situation being exactly the same as in (A).
To see that for $X_{1}$ $C^{r}$-close to $X$, one has $Z_{1}$
$C^{r}$-close to $Z_{0}$ in $\varOmega_{0}$ we use the linearity
of the operator $\varPhi:h\mapsto v$ in Lemma \hyperlink{l1}{1}.
In first place note that since $X_{0}$ and $X_{1}$ are both smooth
and in $\varOmega_{0}$ we have $\widehat{Z_{1}}=\xi_{1}X_{1}+\xi_{0}X_{0}$,
then in $\varOmega_{0}$,
\[
Z_{1}=\widehat{Z_{1}}-\varPhi(\text{div}\,\widehat{Z_{1}})=\xi_{1}X_{1}+\xi_{0}X_{0}-\varPhi(\text{div}\,(\xi_{1}X_{1}+\xi_{0}X_{0}))
\]
is also smooth. On the other hand, using the linearity of the divergence
and that of the operator $\varPhi$, writing $X_{1}=X+(X_{1}-X)$
we have in $\varOmega_{0}$,
\[
Z_{1}=A+B
\]
where
\[
A=\xi_{1}X+\xi_{0}X_{0}-\varPhi\big(\text{div}\,(\xi_{1}X+\xi_{0}X_{0})\big)
\]
and
\[
B=\xi_{1}(X_{1}-X)-\varPhi\big(\text{div}\,(\xi_{1}(X_{1}-X))\big)
\]
Now, since $\widehat{\xi}_{1}=\xi_{1}$ in $\varOmega_{0}$, one has
\[
A=\widehat{Z_{0}}-\varPhi(\text{div}\,\widehat{Z_{0}})=Z_{0}
\]
while (on local charts),
\[
B=\xi_{1}(X_{1}-X)-\varPhi\Bigg(\sum_{i\leq n}\partial_{i}\xi_{1}(X_{1}^{i}-X^{i})\Bigg)
\]
is $C^{r}$-small if $|X_{1}-X|_{r}$ is small. Therefore, $(Z_{1}-Z_{0})|_{\varOmega_{0}}$
is as $C^{r}$-small as wished provided $|X_{1}-X|_{r}$ is small
enough.\smallskip{}

\noindent (\textbf{C}. Construction of $Z_{k}$, $k\geq2$). Proceeding
exactly in the same way as in (B), we let
\[
\widehat{Z_{k}}=\widehat{\xi}_{k+1}X+\xi_{k}X_{k}+\cdots+\xi_{0}X_{0}\in\mathfrak{X}^{r}(M)
\]
where $X_{k}\in\mathfrak{X}_{\mu}^{\infty}(M)$ is as $C^{r}$-close
to $X$ as needed below and 
\[
\widehat{\xi}_{k+1}=\begin{cases}
\xi_{k+1} & \text{in }\varLambda_{0}\cup\cdots\cup\varLambda_{k}\\
1 & \text{elsewhere }
\end{cases}
\]
and then cancel the divergence inside $\varOmega_{k}$ and $\varOmega_{k-1}$
using Lemma \hyperlink{l1}{1}. Reasoning as in (B), we need only
to guarantee that $Z_{k}$ is as $C^{r}$-close to $Z_{k-1}$ as wished
in $\varOmega_{k-1}$. Again, the fact that $\widehat{\xi}_{k}$ and
$\xi_{k}$ coincide in $\varOmega_{k-1}$ guarantees that in this
set,
\[
Z_{k}=\widehat{Z_{k}}-\varPhi(\text{div}\,\widehat{Z_{k}})=Z_{k-1}+B
\]
where $B$ is $C^{r}$-small if $|X_{k}-X|_{r}$ is small, and consequently,
as in (B), $|Z_{k}-Z_{k-1}|_{r;\varOmega_{k-1}}$ is as small as desired
and it straightforward to verify that $Z_{k}$ satisfies (1) - (4)
above.\medskip{}

(Case $\varOmega=\,\,$interior of a collar). The proof is the one
given above, modulo the following simple change: we fix a $C^{2}$
compact collar embedding $\zeta:\partial N\times[0,\epsilon]\hookrightarrow N$,
and define, as in the previous case, a sequence $\varLambda_{k}$
of overlapping ``annuli'' now indexed by $\mathbb{Z}$, forming
an open cover of $\varOmega=\zeta(\partial N\times]0,\epsilon[)$,
with $\varLambda_{k}$ approaching $\partial N$ and $\zeta(\partial N\times\epsilon)$
as $k$ tends to $\infty$ and $-\infty$, respectively. The construction
is then essentially the same, noting that the hipersurfaces $\zeta(\partial N\times\delta)$,
$\delta\in]0,1[$, are of class $C^{2}$, thus the divergence theorem
applies when needed.
\end{proof}
The next result shows that if in Theorem \hyperlink{th3}{3} we want
to have $Z$ satisfying (1) and (3) but are not particularly interested
in having (2) $Z=X$ in $U^{c}$, then the regularity of $Z$ can
be increased to that of $Y$ and it can actually be made $C^{\infty}$
in $U^{c}$. 

\hypertarget{C1}{}
\begin{cor}
\emph{($C^{s,\beta}$ conservative extension with $C^{r}$-closeness)}.
Let $M$ be a manifold as above. Suppose that $K$ is a compact subset
with an open neighbourhood $U\subsetneq M$ such that $U\setminus K$
is connected. Then, given $s\in\mathbb{Z}^{+}\cup\{\infty\}$, $0\leq\beta\leq1$,
and an integer $1\leq r\leq s$, there is an open set $K\subset V\subset U$
and a constant $C=C(r,K,U)>1$ (that of Theorem \hyperlink{th3}{3})
such that: given $X\in\mathfrak{X}_{\mu}^{r}(M)$ and a $C^{r}$-bounded
$Y\in\mathfrak{X}_{\mu}^{s,\beta}(U)$ such that $Y\not\equiv X|_{U}$,
there exists $Z\in\mathfrak{X}_{\mu}^{s,\beta}(M)$ satisfying:

\begin{enumerate}
\item $Z=Y$ in $\overline{V}$;
\item $Z$ is $C^{\infty}$ in a neighbourhood of $U^{c}$;
\item $\left\Vert Z-X\right\Vert _{C^{r}}\leq C\left\Vert Y-X\right\Vert _{C^{r};U}$
\end{enumerate}
Furthermore, if $\beta=0$, then Z is $C^{\infty}$ in $\overline{V}^{c}$.
\end{cor}
\begin{proof}
Fix $\widehat{X}\in\mathfrak{X}_{\mu}^{\infty}(M)$ such that 
\begin{equation}
\|\widehat{X}-X\|_{C^{r}}\leq{\textstyle \frac{1}{2C}}\|Y-X\|_{C^{r};U}
\end{equation}
where $C=C(r,K,U)>1$ is the constant given in Theorem \hyperlink{th3}{3}.
By the observation preceding that result, there is a compact \textbf{$n$-}submanifold\textbf{
$Q\subset U$ }with smooth connected boundary such that $K\subset\text{int\,}Q$
and a vector field $Z_{0}\in\mathfrak{X}_{\mu}^{s,\beta}(M)$ such
that $Z_{0}=Y$ in $Q$, $Z_{0}=\widehat{X}$ in a neighbourhood of
$U^{c}$ and 
\[
\|Z_{0}-\widehat{X}\|_{C^{r}}\leq C\|Y-\widehat{X}\|_{C^{r};U}
\]
By Remark \hyperlink{r4}{4} (Section \hyperlink{se3.1}{3.1}), we
may replace constant $C$ by $C-1$ in the inequality above and get
\[
\|Z_{0}-\widehat{X}\|{}_{C^{r}}\leq(C-1)\|Y-\widehat{X}\|{}_{C^{r};U}\leq(C-1)\left(\|Y-X\|{}_{C^{r};U}+\|X-\widehat{X}\|_{C^{r}}\right)
\]
Combining with (2.5), 
\[
\begin{array}{lll}
\|Z_{0}-X\|{}_{C^{r}} & \leq & \|Z_{0}-\widehat{X}\|{}_{C^{r}}+\|\widehat{X}-X\|_{C^{r}}\\
 & \leq & (C-1/2)\|Y-X\|_{C^{r};U}
\end{array}
\]
Let $V=\text{int}\,Q$. If $\beta>0$, then $Z=Z_{0}$ is the desired
vector field. If $\beta=0$, we get $Z$ as wished applying Theorem
\hyperlink{th5}{5} to $Z_{0}\in\mathfrak{X}_{\mu}^{s}(M)$ and $\varOmega=\overline{V}^{c}$,
the interior of $N=V^{c}$, a compact $n$-submanifold with smooth
connected boundary $\partial N=\partial Q$. 
\end{proof}
\hypertarget{se2.2}{}

\subsection{Conservative local linearization.}

Theorem \hyperlink{th3}{3} can be also used to prove that a divergence-free
vector field can be conservatively $C^{1}$-perturbed to become linearized
near $x\in M$, the perturbation support being a neighbourhood of
$x$ as small as pleased. Although the main application occurs when
the points of $\varSigma$ are singularities of $v$, we formulate
it in the general case. Special care has been taken to find a $\delta$
that directly estimates the permitted variation of the derivative
on all local charts. Observe that given $\epsilon>0$, the same $\delta$
(depending linearly on $\epsilon$) works simultaneously for all divergence-free
vector fields on $M$ in all classes of regularity (c.f. Theorem \hyperlink{th7}{7}
below).

\hypertarget{th6}{}
\begin{thm}
\emph{($C^{s,\beta}$ conservative local linearization - ``Franks
lemma type'')}. Let $M$ be a manifold as above. Then, there is a
constant $\chi>0$ (depending only on the atlas of $M$) such that:
given

- any $\epsilon>0$;

- any $v\in\mathfrak{X}_{\mu}^{s,\beta}(M)$, $s\in\mathbb{Z}^{+}\cup\{\infty\}$,
$0\leq\beta\leq1$;

- any finite set $\varSigma\subset M$;

- any neighbourhood $U$ of $\varSigma$;

- any traceless linear maps $A_{x}\in L(n,\mathbb{R})$, $x\in\varSigma$,
satisfying 
\[
\left\Vert A_{x}-Dv(x)\right\Vert <\chi\epsilon
\]
where $Dv(x)$ is taken in some (reindexed) local chart $(V_{x},\phi_{x})$
around $x$, there exists $Z\in\mathfrak{X}_{\mu}^{s,\beta}(M)$ satisfying:

\begin{enumerate}
\item for each $x\in\varSigma$, on local chart $(V_{x},\phi_{x})$,\\
$Z(y)=v(x)+A_{x}(y-x)$ ~near $x$;
\item $Z=v$ in $U^{c}$;
\item $\left\Vert Z-v\right\Vert _{C^{1}}<\epsilon$.
\end{enumerate}
\end{thm}
\hypertarget{r2}{}
\begin{rem}
(1) implies, for each $x\in\varSigma$, that $Z(x)=v(x)$, and on
local chart $(V_{x},\phi_{x})$, $DZ(x)=A_{x}$ and $Z$ is affine
linear near $x$. 
\end{rem}
\emph{Proof's preview.} The attentive reader will notice that Theorem
\hyperlink{th6}{6} is \emph{not }a particular case of Theorem \hyperlink{th3}{3}.
The result easily reduces to the case $\varSigma$ consists of a single
point. The problem is obviously a local one, the construction being
carried out on some chosen local chart (performing a translation we
may assume that $x=0$). Instead of trying to prove directly that,
for any traceless $A\in L(n,\mathbb{R})$ sufficiently close to $Dv(0)$,
pasting adequately $Y(y)=v(0)+A(y)$ to $v$ on a sufficiently small
neighbourhood $U$ of $x$ (using Theorem \hyperlink{th3}{3}) we
can get a divergence-free vector field $C^{1}$ close to $v$, with
the inherent problem of controlling the growth of constant $C=C(1,K,U)$
as $U$ ``blows down'' to $x,$ we proceed differently and re-scale
to the open unit ball $\mathbb{B}^{n}$, the restrictions of vector
fields $Y$ and $v$ to arbitrarily small balls $\lambda\mathbb{B}^{n}$
(under the action of homotheties $\varPhi_{\lambda}=\lambda^{-1}\mbox{Id}$).
Observing that the $C^{1}$ norm of the vector field 
\[
Y_{\lambda}-v_{\lambda}=\varPhi_{\lambda_{*}}(Y-v)\in\mathfrak{X}_{\mu}^{s,\beta}(\mathbb{B}^{n})
\]
tends to $\|A-Dv(0)\|$ as $\lambda\rightarrow0$, we perform the
pasting on this constant scale, with fixed $K$, $U$ and $C=C(1,K,U)$
and then pullback (scale down) the resulting vector field to the original
real scale, i.e. to a sufficiently small ball $\lambda\mathbb{B}^{n}$,
finally extending it by $v$ to the whole $M$, the non increasing
behaviour of the $C^{1}$ norm under the action of homothetic contractions
guaranteeing the desired conclusion.

\hypertarget{r3}{}
\begin{rem}
In the proof of Theorem \hyperlink{th6}{6} we will need to apply
Theorem \hyperlink{th3}{3} with $M$ an open ball $\eta\mathbb{B}^{n}\subset\mathbb{R}^{n}$.
Obviously, Theorem \hyperlink{th3}{3} remains valid if the manifold
$M$ is instead a connected open subset of $\mathbb{R}^{n}$ equipped
with the trivial one chart atlas $(M,\mbox{Id})$ and both $X,\,Y\in\mathfrak{X}_{\mu}^{s,\beta}(M)$
are $C^{r}$-bounded (see Definition \hyperlink{def1}{1}), $\mu$
being the Lebesgue measure induced by the the canonical volume on
$\mathbb{R}^{n}$. 
\end{rem}
\begin{proof}
(Theorem \hyperlink{th6}{6}). Choose a local chart around each $x\in\varSigma$
and fix on it a small closed ball $\overline{B_{x}}$ centred at $x$
(we identify $x$ with its image on the chart), so that these balls
have disjoint preimages on $M$ and are contained in $U$. Changing
$U$ by the union of the interiors of these $\#\varSigma$ balls it
is immediate that the proof reduces to the case of $\varSigma$ consisting
of a single point $x$. Let $d=d(1,\mbox{max}_{i,j\leq m}\left\Vert \phi_{ji}\right\Vert _{C^{2}})\geq1$
be the constant controlling the potential magnification of the local
$C^{1}$ norm of a vector field under the chart transitions of the
atlas (see , Section \hyperlink{3.1c}{3.1(c)}). Get constant $C=C(1,{\textstyle \frac{1}{3}}\mathbb{D}^{n},{\textstyle \frac{2}{3}}\mathbb{B}^{n})$
given by Theorem \hyperlink{th3}{3} for $M=\mathbb{B}^{n}$ taking
Remark \hyperlink{r3}{3} into consideration and let $\chi=1/(Cd)$.
Take a local chart $(W,\,\phi)$ around $x$. Performing a translation
we may assume that $\phi(x)=0\in\mathbb{R}^{n}$. Take $\eta>0$ such
that $\eta\mathbb{D}^{n}\subset\phi(W)$ and $\phi^{-1}(\eta\mathbb{D}^{n})\subset U$.
To simplify the notation we still denote by $v$ the vector field
$\phi_{*}v|_{W}\in\mathfrak{X}_{\mu}^{s,\beta}(\phi(W))$ (recall
that the atlas is regular (see the Convention, Section \hyperlink{se2}{2}),
hence this local chart expression of $v$ is $C^{1}$-bounded; $\mu$
is now the Lebesgue measure on $\mathbb{R}^{n}$). Fix any traceless
$A\in L(n,\mathbb{R})$ such that
\[
\left\Vert A-Dv(0)\right\Vert <\chi\epsilon
\]
(recall that $\phi(x)=0$ and $Dv(0)$ is taken on local chart $(W,\phi)$).
Define on $\eta\mathbb{B}^{n}$, 
\[
Y(y)=v(0)+A(y)-v(y)
\]

\noindent \noun{Homothety trick} - (Step 1). \noun{Re-scaling to the
unit scale.} For each $0<\lambda<\text{min}(1,\eta)$, re-scale $Y|_{\lambda\mathbb{B}^{n}}$
to the unit ball $\mathbb{B}^{n}$
\[
Y_{\lambda}=\left(\lambda^{-1}\mbox{Id}\right)_{*}Y|_{\lambda\mathbb{B}^{n}}\in\mathfrak{X}_{\mu}^{s,\beta}(\mathbb{B}^{n})
\]

\noindent \noun{Claim.} $\left\Vert Y_{\lambda}\right\Vert _{C^{1};\mathbb{B}^{n}}\xrightarrow[\lambda\rightarrow0]{}\left\Vert A-Dv(0)\right\Vert .$ 

\noindent Recall that $\left\Vert Y_{\lambda}\right\Vert _{C^{1};\mathbb{B}^{n}}=\mbox{max}\big(\left\Vert Y_{\lambda}\right\Vert _{C^{0};\mathbb{B}^{n}},\,\underset{x\in\mathbb{B}^{n}}{\text{sup}}\left\Vert DY_{\lambda}\right\Vert \big)$.

\noindent (a) The derivative is unchanged by the action of the homothety,
\[
DY_{\lambda}(y)=DY(\lambda y)=A-Dv(\lambda x)\;\mbox{ for all \ensuremath{y\in\mathbb{B}^{n}}}
\]
therefore, since $v$ is $C^{1}$,
\[
\underset{y\in\mathbb{B}^{n}}{\mbox{sup}}\left\Vert DY_{\lambda}\right\Vert =\underset{y\in\lambda\mathbb{B}^{n}}{\mbox{sup}}\left\Vert DY\right\Vert \xrightarrow[\lambda\rightarrow0]{}\left\Vert A-Dv(0)\right\Vert 
\]

\noindent (b) as for the $C^{0}$ norm,
\[
\left\Vert Y_{\lambda}\right\Vert _{C^{0};\mathbb{B}^{n}}=\lambda^{-1}\left\Vert Y\right\Vert _{C^{0};\lambda\mathbb{B}^{n}}\xrightarrow[\lambda\rightarrow0]{}\left\Vert A-Dv(0)\right\Vert 
\]

\noindent since
\[
\begin{array}{lll}
\lambda^{-1}\left\Vert Y\right\Vert _{C^{0};\lambda\mathbb{B}^{n}} & = & \underset{y\in\lambda\mathbb{B}^{n}}{\mbox{sup}}\lambda^{-1}\left|v(0)+A(y)-v(y)\right|=\\
 & = & \underset{y\in\lambda\mathbb{B}^{n}}{\mbox{sup}}\left|\frac{v(0)+Dv(0;y)-v(y)}{\lambda}+\frac{A(y)-Dv(0;y)}{\lambda}\right|\xrightarrow[\lambda\rightarrow0]{}\left\Vert A-Dv(0)\right\Vert 
\end{array}
\]
as it is immediate to verify: the fraction on the left converges to
$0\in\mathbb{R}^{n}$ as $\lambda\rightarrow0$, while
\[
\underset{y\in\lambda\mathbb{B}^{n}}{\mbox{sup}}\frac{\left|A(y)-Dv(0;y)\right|}{\lambda}=\underset{y\in\mathbb{B}^{n}}{\mbox{sup}}\left|A(y)-Dv(0;y)\right|=\left\Vert A-Dv(0)\right\Vert 
\]
Therefore, for $0<\lambda<\eta$ small enough
\[
\left\Vert Y_{\lambda}\right\Vert _{C^{1};\mathbb{B}^{n}}<\chi\epsilon
\]

\smallskip{}

\noindent (Step 2). \noun{Performing the pasting. }Letting $X\equiv0$
on $\mathbb{B}^{n}$, by Theorem \hyperlink{th3}{3} (and Remark \hyperlink{r3}{3}),
there is $Z_{1}\in\mathfrak{X}_{\mu}^{s,\beta}(\mathbb{B}^{n})$ such
that
\[
\begin{cases}
Z_{1}=Y_{\lambda}\quad\mbox{in}\;\frac{1}{3}\mathbb{B}^{n}\\
Z_{1}=0\quad\mbox{\,\,\ in}\;\mathbb{B}^{n}\setminus\frac{2}{3}\mathbb{B}^{n}\\
\left\Vert Z_{1}\right\Vert _{C^{1};\mathbb{B}^{n}}\leq C\left\Vert Y_{\lambda}\right\Vert _{C^{1};\mathbb{B}^{n}}<C\chi\epsilon=\epsilon/d
\end{cases}
\]

\smallskip{}

\noindent (Step 3). \noun{Scaling down to the real scale.} Pullback
$Z_{1}$ to the ``real scale'' defining
\[
Z_{0}=\left(\lambda^{-1}\mbox{Id}\right)^{*}Z_{1}\in\mathfrak{X}_{\mu}^{s,\beta}(\lambda\mathbb{B}^{n})
\]
compactly supported in $\lambda\mathbb{B}^{n}$. Extend $Z_{0}$ by
$0$ to the whole $\eta\mathbb{B}^{n}$ and define on this set, $Z=Z_{0}+v$.
Then, $Z=v(0)+A$ in $\frac{\lambda}{3}\mathbb{B}^{n}$ and $Z=v$
in $\mathbb{\eta B}^{n}\setminus\frac{2\lambda}{3}\mathbb{B}^{n}$.
Since $\lambda<1$, $Z_{1}\mapsto Z_{0}$ is a homothetic contraction,
thus the $C^{1}$ norm does not increase and 
\[
\left\Vert Z-v\right\Vert _{C^{1};\eta\mathbb{B}^{n}}=\left\Vert Z_{0}\right\Vert _{C^{1};\lambda\mathbb{B}^{n}}\leq\left\Vert Z_{1}\right\Vert _{C^{1};\mathbb{B}^{n}}<\epsilon/d
\]
We finally get the desired $Z\in\mathfrak{X}_{\mu}^{s,\beta}(M)$
extending the pullback $\phi^{*}(Z)$ by $v$ to the whole $M$. Note
that $Z-v\in\mathfrak{X}_{\mu}^{s,\beta}(M)$ is compactly supported
inside $\phi^{-1}(\eta\mathbb{B}^{n})$, thus the global $C^{1}$
norm of $Z-v$ satisfies
\[
(3)\text{ \,}\left\Vert Z-v\right\Vert _{C^{1}}\leq d\left\Vert Z_{0}\right\Vert _{C^{1};\lambda\mathbb{B}^{n}}<\epsilon
\]
and it is immediate to verify that (1) and (2) are also satisfied.
\end{proof}
\hypertarget{se2.3}{}

\subsection{Conservative pasting - Proof of Theorem 1.}

(\emph{Preview})\emph{.} Using Lemma \hyperlink{l3}{3} (Section \hyperlink{se5.3}{5.3})
and the existence of collars for manifolds with boundary, fix $W_{0}$
and $W_{1}$, two compact $n$-submanifolds with $C^{\infty}$ boundary
such that
\[
K\subset\mbox{int}\,W_{0}\mbox{, }\quad W_{0}\subset\mbox{int}\,W_{1},\mbox{ }\quad W_{1}\subset U,\mbox{ }\quad\mbox{\ensuremath{\varOmega}:=(int}\,W_{1})\setminus W_{0}\mbox{ \,\ is connected}
\]
The transition from $Y$ to $X$ will take place inside the open set
$\mbox{\ensuremath{\varOmega}}$. Fix $\xi\in C^{\infty}(M;[0,1])$
such that $\xi=1$ in a small neighbourhood of $W_{0}$ and $\xi=0$
in a small neighbourhood of $(\mbox{int}\,W_{1})^{c}$. Now given
any $X,\,Y$ as in the statement let
\[
\begin{cases}
w=\xi Y+(1-\xi)X & \mbox{ in }U\\
w=X & \mbox{ in }U^{c}
\end{cases}
\]
Note that $w\in\mathfrak{X}^{s}(M)$ since $\xi=0$ in a neighbourhood
of $U^{c}$ and $Y$ is defined and of class $C^{s}$ on $U$. Since
both $X\in\mathfrak{X}_{\mu}^{s}(M)$ and $Y\in\mathfrak{X}_{\mu}^{s}(U)$
are divergence-free, 
\[
h:=\mbox{div}\,w\in C^{s}(M)\quad\text{and}\quad h\text{ \,is \,}C^{r}\text{-small if \,}Y-X|_{U}\text{ \,is \,}C^{r}\text{-small}
\]
as $h=\mbox{div\,}X=0$ in a neighbourhood of $U^{c}$ and (on local
charts),
\begin{equation}
h=\sum_{i=1}^{n}(\partial_{i}\xi)(Y^{i}-X^{i})\;\,\mbox{in }U
\end{equation}
Clearly, $h$ is (compactly) supported inside $\varOmega$. In order
to get $Z_{0}\in\mathfrak{X}_{\mu}^{s}(M)$ satisfying (\hyperlink{1}{1})
and (\hyperlink{2}{2}), it is enough to find $v\in\mathfrak{X}^{s}(M)$
supported inside $\varOmega$ such that
\[
\mbox{div}\,v=h=\mbox{div}\,w
\]
and then let $Z_{0}=w-v,$ thus canceling the divergence of $w$ inside
the ``transition annulus'' $\varOmega$, while keeping $w$ unaltered
outside that open set (in particular, $Z_{0}=w=Y$ in a neighbourhood
of $W_{0}$ and $Z_{0}=w=X$ in a neighbourhood of $(\mbox{int}\,W_{1})^{c}\supset U^{c}$).
Since the smooth scalar function $\xi$ is fixed, by (2.6) the $C^{r}$
norm of $h$ is linearly bounded by that of $Y-X|_{U}$,
\begin{equation}
\left|h\right|_{r}\leq n2^{r}\bigl|\xi\bigr|_{r+1}\,\left|Y-X\right|_{r;U}
\end{equation}
and it can be shown that (3) holds (see Section \hyperlink{se3}{3}).
The crucial facts that guarantee the existence of canceling vector
field $v$ are: (a) the connectedness of $\varOmega$, (b) $\mbox{supp}\,h\subset\varOmega$
and (c) $\int_{\varOmega}h\omega=0$, this equality following readily
from the divergence theorem since $X,\,Y$ are divergence-free vector
fields, $w$ coincides with $Y$ and $X$ in $\partial W_{0}$ and
$\partial W_{1}$ (respectively) and $\partial\overline{\varOmega}=\partial W_{0}\sqcup\partial W_{1}$,
thus
\[
\begin{array}{lll}
\int_{\varOmega}h\omega=\int_{\partial\varOmega}w\lrcorner\,\omega & = & -\int_{\partial W_{0}}Y\lrcorner\,\omega+\int_{\partial W_{1}}X\lrcorner\,\omega\\
 & \,\\
 & = & -\int_{W_{0}}\text{(div}\,Y)\omega+\int_{W_{1}}(\text{div}\,X)\omega=-0+0=0
\end{array}
\]
The actual construction of $v$ uses the global-to-local reduction
technique originally devised by Moser in \cite{MO}, essentially aiming
to solve, under condition (c), equation $\mbox{det}\,Df=1+h$ on closed
manifolds. We shall follow a complete presentation of the transposition
of this technique to the solution of $\mbox{div}\,u=h$ on $\varOmega\subset\mathbb{R}^{n}$
(under specific support premises) given by Csató, Dacorogna and Kneuss
\cite[p.184-188]{CDK}. The smoothing of $Z_{0}$ inside the transition
annulus $\varOmega$ is the last step of the construction.\smallskip{}

As a byproduct of the proof below together with the estimates in Sections
\hyperlink{se3.1}{3.1} and \hyperlink{se3.2}{3.2}, we obtain the
following useful result on the solutions to the equation $\text{div\,}v=h$,
with control of support (this is applied in the proofs of Theorems
\hyperlink{th4}{4} and \hyperlink{th5}{5}). The linearity of the
operator $\varPhi:h\mapsto v$ is immediate to check from its construction
(c.f. \cite[p.184-188]{CDK}). Another important aspect is that the
operator is universal i.e. $v$ has always the same regularity as
$h$ (the construction being independent of $r$ and $\alpha$) and
its $C^{r,\alpha}$ norm can be estimated in terms of that of $h$
times a constant, i.e. the restriction of linear operator $\varPhi$
to the subspace of $\mathcal{A}$ consisting of those functions $h$
that are of class $C^{r,\alpha}$ is bounded for the $\left\Vert \cdot\right\Vert _{C^{r,\alpha}}$
norm. 

\hypertarget{l1}{}
\begin{lem}
Let $M$ be a manifold as above. Suppose that $\varOmega_{1}$, $\varOmega$
are two connected open subsets with $\overline{\varOmega_{1}}\subset\varOmega$.
Then, there exists a linear operator $\varPhi:\mathcal{A}\rightarrow\mathcal{B}:\,h\mapsto v$,
satisfying $\text{\emph{div}\,}v=h$, where
\[
\begin{array}{l}
\mathcal{A}=\{h\in C^{1}(M):\,\int_{\varOmega}h\,\omega=0\text{ \,and \,}\text{\emph{supp}}\,h\subset\varOmega_{1}\}\\
\,\\
\mathcal{B}=\{v\in\mathfrak{X}^{1}(M):\,\text{\emph{supp}}\,v\subset\varOmega\}.
\end{array}
\]
Furthermore, if $h$ is of class $C^{r,\alpha}$, $r\in\mathbb{Z}^{+}$,
$0\leq\alpha\leq1$, then $v$ is $C^{r,\alpha}$ and there is a constant
$C=C(r,\varOmega_{1},\varOmega)\geq1$ such that 
\[
\left\Vert v\right\Vert _{C^{r,\alpha}}\leq C\left\Vert h\right\Vert _{C^{r,\alpha}}
\]
\end{lem}
\begin{proof}
(Theorem \hyperlink{th1}{1}). According to the Preview, it remains
to define $\varOmega$ and $\xi$ precisely and then solve
\[
\mbox{div}\,v=h\mbox{, }\quad v\in\mathfrak{X}^{s}(M)\mbox{ supported inside }\varOmega
\]
The existence of constant $C=C(r,K,U)$ satisfying (\hyperlink{3}{3})
is proved in Section \hyperlink{se3}{3}. We start by carefully constructing
$\varOmega$ and an auxiliary domain $\varOmega_{1}$, which is needed
in our approach.

(A) \noun{Construction of $\varOmega$, $\varOmega_{1}$, $V$ and
$w$.} Using Lemma \hyperlink{l3}{3} (Section \hyperlink{se5.3}{5.3}),
fix a compact $n$-submanifold $P$ with \emph{connected} $C^{\infty}$
boundary such that $K\subset\mbox{int}\,P$ and $P\subset U$. By
the existence of collars for $\partial P$ \cite[p.113]{HI}, there
are four smoothly isotopic (nested) manifolds $P_{i\leq4}$ satisfying
\[
K\subset\mbox{int}\,P_{1}\mbox{, }\qquad P_{i}\subset\mbox{int}\,P_{i+1}\quad(i\leq3),\mbox{ }\qquad P_{4}=P
\]
and such that 
\[
\varOmega:=\mbox{(int}\,P_{4})\setminus P_{1}\qquad\mbox{and}\qquad\varOmega_{1}:=\mbox{(int}\,P_{3})\setminus P_{2}
\]
are both diffeomorphic to $\partial P\times]0,1[$, hence connected
open sets. Exactly as described in the Preview, fix a scalar function
$\xi$ for $W_{0}=P_{2}$ and $W_{1}=P_{3}$ (the same for all $X$
and $Y$) and define $w$ and $h$ accordingly. Clearly $h\in C^{s}(M)$
is supported inside $\varOmega_{1},$ $\overline{\varOmega_{1}}\subset\varOmega$
and $\int_{\varOmega}h\omega=0$. We set $V=\mbox{int}\,P_{1}$.

(B) \noun{Finding divergence-canceling vector field $v$.} In order
to find $v\in\mathfrak{X}^{s}(M)$ supported inside $\varOmega$ and
satisfying $\mbox{div}\,v=h$, we may now apply the procedure in \cite[p.184-188]{CDK},
reducing this problem to the solution of finitely many local equations\hypertarget{eq2.8}{}
\begin{equation}
\mbox{div}\,v_{j}=h_{j},\quad v_{j}\in\mathfrak{X}^{s}(Q_{j})
\end{equation}
with $v_{j}$ compactly supported inside the open cube $Q_{j}\subset\mathbb{R}^{n}$.
The construction in \cite{CDK} carries almost verbatim to our closed
manifold $M$, as the integrals involved in the definition of the
auxiliary functions $h_{j}$ are invariant under chart transition
(see below).

Briefly, since $\overline{\varOmega_{1}}\subset\varOmega$ is compact,
it can be covered by finitely many small open sets $U_{j}\subset\varOmega$,
$0\leq j\leq N$, $N\geq3$, each of them intersecting $\varOmega_{1}$,
such that the image of each $U_{j}$ on some (reindexed) local chart
$(V_{j},\,\phi_{j})$ is an open cube $Q_{j}\subset\phi_{j}(V_{j})\subset\mathbb{R}^{n}$
of volume $\leq1$.\footnote{This fact will be used in Section \hyperlink{3.1c}{3.1(c)}.}
Clearly, $N$ depends only on $\varOmega_{1}$ and $\varOmega$ and
thus ultimately only on $K$ and $U$. Auxiliary functions $h_{j}\in C^{s}(M)$
are now constructed exactly as in \cite[p.185, Lemma 9.9]{CDK}. These
are well defined since the atlas is volume preserving, thus implying
that all integrals of scalar functions involved \cite[p.187]{CDK}
are invariant under chart transition (these appear in the constants
$\lambda_{k}$, see Section \hyperlink{3.1b}{3.1(b)}. The scalar
functions $h_{j}$ satisfy \cite[Lemma 9.9]{CDK}
\[
h=\underset{j=0}{\overset{N}{\sum}}h_{j},\qquad\mbox{supp}\,h_{j}\subset U_{j}\subset\varOmega,\qquad\int_{U_{j}}h_{j}\omega=0
\]
On local chart $(V_{j},\,\phi_{j})$, 
\[
\int_{Q_{j}}h_{j}=0,\qquad\mbox{supp}\,h_{j}\subset Q_{j}\subset\mathbb{R}^{n}
\]
Each local equation (\hyperlink{eq2.8}{2.8}) is now solved by \cite[p.185, Lemma 9.8]{CDK}
(which is valid for arbitrary open cubes, see Footnote 5) and the
pullback $\phi_{j}^{*}v_{j}$, still denoted by $v_{j}$, is extended
by $0$ to the whole $M$. As $h=\sum_{j=0}^{N}\mbox{div}\,v_{j}=\mbox{div\,(}\sum_{j=0}^{N}v_{j})$
and $\mbox{supp}\,v_{j}\subset\varOmega$, $v=\sum_{j=0}^{N}v_{j}$
is the desired vector field. Observe that, by construction, $h_{j}$,
$v_{j}$ and finally $v$ are $C^{s}$ if $h$ is $C^{s}$ (i.e. if
$X,\,Y$ are $C^{s}$, see Section \hyperlink{se3}{3}). We now have
$Z_{0}=w-v\in\mathfrak{X}_{\mu}^{s}(M)$ satisfying (1) and (2). Observe
that the above procedure actually gives a construction of the operator
$\varPhi$ in Lemma \hyperlink{l1}{1}, i.e. $v=\varPhi(\text{div}\,w)$.
Still, by construction, if $Y=X|_{U}$ then $Z_{0}=X$ (see Lemma
\hyperlink{l1}{1} above) hence $Z=X$. Otherwise, by Remark \hyperlink{r4}{4}
(Section \hyperlink{se3.1}{3.1}), the estimate $\left\Vert Z_{0}-X\right\Vert _{C^{r}}\leq C\left\Vert Y-X\right\Vert _{C^{r};U}$
is still valid with constant $C$ replaced by $C-1$ and we finally
get $Z$ still satisfying (1) - (3) and smooth in 
\[
\Delta=\{x\in M:\,Z(x)\neq X(x),\:Y(x)\}
\]
applying Theorem \hyperlink{th5}{5} to $Z_{0}$ and $\varOmega$
(this set being the interior of a compact collar of $\partial P$).
\end{proof}
\hypertarget{se3}{}

\section{Linear bound on $C^{r,\alpha}$ norms }

\hypertarget{se3.1}{}

\subsection{The $C^{r}$ case }

Instead of the standard Whitney $C^{r}$ norm $\left\Vert \cdot\right\Vert _{C^{r}}$,
we adopt the equivalent but more convenient norm $\left|\cdot\right|_{r}$
defined in Section \hyperlink{se5.1}{5.1}. Then, estimate (3) in
Theorem \hyperlink{th1}{1} is proved letting $C=n^{(r+1)/2}C'+1$
and finding a constant $C'=C'(r,K,U)$ for which 
\begin{equation}
\left|Z_{0}-X\right|_{r}\leq C'\left|Y-X\right|_{r;U}
\end{equation}
(clearly, $C=C(r,K,U)$ since $n=\mbox{dim}\,M$ is fixed).

\hypertarget{r4}{}
\begin{rem}
Note that the estimate (3) in Theorem \hyperlink{th1}{1} will still
be valid if one replaces $C$ by $C-1$ (as a consequence of adding
$+1$ in the definition of $C$). This is used at the end of the proof
of Theorem \hyperlink{th1}{1} (in the smoothing step). The same observation
holds for Theorem \hyperlink{th3}{3} (used in Corollary \hyperlink{C1}{1}). 
\end{rem}
As
\[
\begin{array}{c}
\left|Z_{0}-X\right|_{r}=\left|w-v-X\right|_{r}\leq\left|w-X\right|_{r}+\left|v\right|_{r}\\
\,\\
\left|w-X\right|_{r}=\left|\xi(Y-X)\right|_{r;U}\leq2^{r}\left|\xi\right|_{r}\left|Y-X\right|_{r;U}
\end{array}
\]
it is enough to find a constant $C_{0}=C_{0}(r,K,U)>0$ such that
$\left|v\right|_{r}\leq C_{0}\left|Y-X\right|_{r;U}$ and let $C'=2^{r}\left|\xi\right|_{r}+C_{0}$
(as $\left|\xi\right|_{r}$ depends only on $r$ and $\varOmega$
and thus ultimately only on $r$, $K$ and $U$). 

We will obtain a finite chain of linear bounds with constants $C_{1}$,
$C_{2}$, $C_{3}$ depending only on $r$, $K$ and $U$, finally
leading to the desired inequality.

\hypertarget{3.1a}{}(a) $\left|h\right|_{r}\leq C_{1}\left|Y-X\right|_{r;U}$.
From the local chart expression of $h$ (see (2.6) and (2.7) in the
Preview, Section  \hyperlink{se2.3}{2.3}), it follows that this inequality
holds for $C_{1}=n2^{r}\left|\xi\right|_{r+1}$. Thus,$\left|\xi\right|_{r+1}$
depends only on $r$, $K$ and $U$, so does $C_{1}$.

\hypertarget{3.1b}{}(b) $\bigl|h_{j}\bigr|_{r}\leq C_{2}\left|h\right|_{r}$.
Following the reasoning in \cite[Section 9.3, p.184-188]{CDK} transposed
to $M$, fix $\psi_{j},\,\eta_{k}\in C^{\infty}(M;[0,1])$ as in Lemma
9.9. Note that $\psi_{j},\,\eta_{k}$ depend ultimately only on $K$
and $U$. Let
\[
d_{1}=\underset{0\leq j\leq N}{\mbox{max}}\left|\psi_{j}\right|_{r},\qquad d_{2}=\underset{1\leq k\leq N}{\mbox{max}}\left|\eta_{k}\right|_{r}
\]
By definition, $h_{j}=h\psi_{j}+\sum_{k=1}^{N}\lambda_{k}A_{k}^{j}\eta_{k}$
(see the proof of Lemma 9.9 in \cite[p.185-188]{CDK}) where each
$A_{k}^{j}$ (depending on the sequence $U_{0},\ldots,U_{N}$) is
either $-1$, $0$ or $1$ and the $\lambda_{k}$'s are the constants
solving $\sum_{k=1}^{N}\lambda_{k}A_{k}^{j}=\int_{\varOmega}h\psi_{j}$,
$0\leq j\leq N$. In order to find the $\lambda_{k}$'s, we solve
the $N$ simultaneous equations corresponding to $1\leq j\leq N$,
as matrix $E$ obtained from $(N+1)\times N$ matrix $A=(A_{k}^{j})$
truncating its first line is actually invertible and the solutions
thus obtained automatically satisfy the equation corresponding to
$j=0$. Finding $\lambda_{k}$ by Cramer rule, $\lambda_{k}=\left|B\right|/\left|E\right|,$
and expanding determinant $\left|B\right|$ along the $k$-th column
(knowing that $A_{k}^{k}=1$, $A_{k}^{j}=-1$ or $0$ if $j<k$, $A_{k}^{j}=0$
if $j>k$ and each column of $E$ contains, at most, two nonzero entries),
we immediately get, on the chart containing the cube $U_{j}$ (recalling
that $N\geq3$),
\[
\left|\lambda_{k}\right|\leq N2^{N-3}\,\underset{0\leq j\leq N}{\mbox{max}}\left|\int_{\varOmega}h\psi_{j}\right|\leq N2^{N-3}\mbox{meas\,}\varOmega\left|h\right|_{r}
\]
\[
\left|h_{j}\right|_{r}\leq\left|h\psi_{j}\right|_{r}+N\underset{0\leq j\leq N}{\mbox{max}}\left|\lambda_{k}A_{k}^{j}\eta_{k}\right|_{r}\leq C_{2}\left|h\right|_{r}
\]
where $C_{2}=C_{2}(r,K,U)=2^{r}d_{1}+N^{2}2^{N-3}d_{2}\mbox{meas}\,\varOmega$.

\hypertarget{3.1c}{}(c) $\bigl|v_{j}\bigr|_{r}\leq C_{3}\left|h_{j}\right|_{r}$.
Recall that $v_{j}$ is found on local chart $(V_{j},\,\phi_{j})$
as the solution of (2.8) given by \cite[Lemma 9.8, p.185]{CDK} and
then extending its pullback by $0$ to the whole $M$. Clearly, Lemma
9.8 \cite{CDK} holds for each cube $Q_{j}\subset\mathbb{R}^{n}$.\footnote{The proof of Lemma 9.8 in \cite[p.185]{CDK} becomes valid for $Q_{j}$
performing the obvious translation of the cube and replacing $\xi$
by $\widehat{\xi}_{j}\in C_{0}^{\infty}(]0,\,\rho_{j}[)$, $\rho_{j}=\mbox{(vol}\,Q_{j})^{1/n}\leq1$,
satisfying $\int_{0}^{\rho_{j}}\widehat{\xi}_{j}=1$. Each $\widehat{\xi}_{j}$
is fixed and depends only on $\mbox{vol}\,Q_{j}$, hence ultimately
only on $K$ and $U$.} Since $\mbox{vol}\,Q_{j}\leq1$, a simple induction argument over
the dimension $n$ (carried on the modified proof of \cite[Lemma 9.8]{CDK},
see Footnote 5) shows that, on local chart $(V_{j},\,\phi_{j})$,
\[
\left|v_{j}\right|_{r}\leq\big(2^{r}\big|\widehat{\xi}_{j}\big|_{r}\big)^{n}\left|h_{j}\right|_{r}
\]
Now, in order to get the global $C^{r}$ norm of $v_{j}$ we need
to take into account the potential magnification of these local norms
under chart transitions $(\phi_{ji})_{i,j\leq m}$. Since the transitions
between the chart expressions of a vector field are of the form 
\[
X_{j}|_{\phi_{j}(V_{i}\cap V_{j})}=\phi_{ji_{*}}X_{i}|_{\phi_{i}(V_{i}\cap V_{j})}
\]
it is easily seen that there is a constant 
\[
d=d(r,\mbox{max}_{i,j\leq m}\bigl|\phi_{ji}\bigr|_{r+1})\geq1
\]
such that 
\[
\bigl|X_{j}|_{\phi_{j}(V_{i}\cap V_{j})}\bigr|_{r}\leq d\bigl|X|_{\phi_{i}(V_{i}\cap V_{j})}\bigr|_{r}
\]
for any $i,\,j\leq m$. The global $C^{r}$ norm of $v_{j}$ can then
be estimated by
\[
\left|v_{j}\right|_{r}\leq C_{3}\left|h_{j}\right|_{r}\mbox{,}\quad\mbox{where }\,C_{3}=d(2^{r}d_{0})^{n},\quad\mbox{ }d_{0}=\underset{0\leq j\leq N}{\mbox{max}}\big|\widehat{\xi}_{j}\big|_{r}
\]
As the atlas is fixed, we actually have $C_{3}=C_{3}(r,K,U)$. 

(d) Finally, $v=\sum_{j=0}^{N}v_{j}$, hence $\left|v\right|_{r}\leq(N+1)\mbox{max}_{_{0\leq j\leq N}}\bigl|v_{j}\bigr|_{r}$,
therefore,
\[
\left|v\right|_{r}\leq(N+1)C_{1}C_{2}C_{3}\left|Y-X\right|_{r;U}
\]
As $N+1$, $C_{1}$, $C_{2}$ and $C_{3}$ depend only on $r,$ $K$,
$U$, the desired constant is $C_{0}=(N+1)C_{1}C_{2}C_{3}$.

\hypertarget{se3.2}{}

\subsection{The $C^{r,\alpha}$ case, $0<\alpha\leq1$.}

In first place we note that a direct inspection of the construction
given in the proof of Theorem \hyperlink{th1}{1} of the operator
$\varPhi$ in Lemma \hyperlink{l1}{1} reveals that the resulting
vector field $Z=w-\varPhi(\text{div}\,w)$ is of class $C^{s,\beta}$
if $X$ and $Y$ are $C^{s,\beta}$, $s\in\mathbb{Z}^{+}$, $0\leq\beta\leq1$.
The proof of Theorem \hyperlink{th3}{3} is that of Theorem \hyperlink{th1}{1}
modulo a few changes needed to get estimate (3) that we now indicate.
To simplify the estimates, it is preferable to work exclusively with
the following $C^{r,\alpha}$ norm, which is equivalent to the usual
Whitney-Hölder $C^{r+\alpha}$ norm $\left\Vert \cdot\right\Vert _{C^{r,\alpha}}$
(see Section \hyperlink{se5.1}{5.1} for the notation):
\[
\left|X\right|_{r,\alpha;U}=\underset{\substack{i,j;\,|\sigma|=r}
}{\mbox{max}}\left(\left|X\right|_{r;U},\,\left[\partial^{\sigma}X_{j}^{i}\right]_{\alpha;\phi_{j}(V_{j}\cap U)}\right)
\]
, the $\alpha$-Hölder seminorm $\left[h\right]_{\alpha;D}\text{}$
of a scalar function $h$ on a domain $D$ (with at least two points)
being defined in the usual way. On local charts, this is also equivalent
to the $C^{r,\alpha}$ norm adopted in \cite[p.336]{CDK}, which serves
as a reference for the estimates invoked below. We will need reasonable
estimates for the Hölder norms of the product and composition of functions
defined on open subsets $A\subset M$, and these exist provided that,
(i) on every local chart, the domain $\phi_{i}(V_{i}\cap A$) of each
function involved is a Lipschitz set (see e.g \cite[p.338, 366, 369]{CDK})
and (ii) these functions and their derivatives up to order $r$ extend
continuously to the boundaries of these domains (we generically denote
the space of $C^{r,\alpha}$ functions on $A$ satisfying (ii) by
$C^{r,\alpha}(\overline{A})$). With these two conditions we also
guarantee the respective inclusion of Hölder spaces: if $r+\alpha\leq s+\beta$
where $0\leq r\leq s$ are integers and $0\leq\alpha,\beta\leq1$,
then $C^{s,\beta}(\overline{A})\subset C^{r,\alpha}(\overline{A})$
and there is a constant $C=C(s,A)>0$ such that $\left|\cdot\right|_{r,\alpha;A}\leq C\left|\cdot\right|_{s,\beta;A}$
\cite[p.342]{CDK}.

Instead of the estimate at the end of Section \hyperlink{se5.1}{5.1},
we now use for the norm of the product of functions in $C^{r,\alpha}(\overline{A})$
(see e.g. \cite[p. 366]{CDK}), 
\begin{equation}
\left|hX\right|_{r,\alpha;A}\leq C(r,A)\left|h\right|_{r,\alpha;A}\left|X\right|_{r,\alpha;A}
\end{equation}
provided each open set $\phi_{j}(V_{j}\cap A)$ is Lipschitz. At first
sight, this may seem problematic for the estimates involving the vector
field $Y$, whose domain $U$ may not intersect the local charts in
Lipschitz sets (also, while $C^{r}$-bounded, $Y$ may fail to satisfy
condition (ii)).\textbf{ }This difficulty is circumvented by the following
simple observation (replacing steps (a) - (c) in Section \hyperlink{se3.1}{3.1}):

(a') following the proof of Theorem \hyperlink{th1}{1}, $w=X$ in
a neighborhood of $(\text{int}\,P)^{c}$, thus
\[
\left|w-X\right|_{r,\alpha}=\left|w-X\right|_{r,\alpha;\text{int}\,P}=\left|\xi(Y-X)\right|_{r,\alpha;\text{int}\,P}
\]

Now, $P$ is a smooth compact $n$-submanifold with boundary and since
the atlas is regular so are the closures $\overline{V_{i}}$ of the
chart domains (these are embedded $\mathbb{D}^{n}$'s). Thus each
open set $\phi_{i}(V_{i}\cap\text{int}\,P)$ is Lipschitz and so are
the domains $\phi_{i}(V_{i}\cap V_{j})$ of the transition maps $\phi_{ji}$.
Therefore (as $P$ and $\xi$ depend only on $K$ and $U$), 

\[
\begin{array}{lll}
\left|w-X\right|_{r,\alpha} & \leq & C(r,K,U)\left|\xi\right|_{r,\alpha}\left|Y-X\right|_{r,\alpha;\text{int}\,P}\\
 & = & C(r,\alpha,K,U)\left|Y-X\right|_{r,\alpha;\text{int}\,P}
\end{array}
\]
and
\[
\begin{array}{lll}
\left|h\right|_{r,\alpha} & = & \left|h\right|_{r,\alpha;\text{int}\,P}\\
 & \leq & C(r,K,U)\left|\xi\right|_{r+1,\alpha}\left|Y-X\right|_{r,\alpha;\text{int}\,P}\\
 & = & C(r,\alpha,K,U)\left|Y-X\right|_{r,\alpha;\text{int}\,P}
\end{array}
\]

From now on we need not concern with condition (ii) anymore, as it
is immediate to verify that all functions involved satisfy it.

(b') the finitely many auxiliary functions $\xi$, $\psi_{j}$, $\eta_{j}$
are defined on the whole $M$, thus using (3.2) one gets the local
estimate (on the chart containing the cube $\phi_{j}(V_{j})$),
\[
\left|h_{j}\right|_{r,\alpha}\leq C(r,\alpha,K,U)\left|h\right|_{r,\alpha}
\]

(c') the auxiliary functions involved in the construction of the compactly
supported solution to $\text{div\,}v_{j}=h_{j}$ on the cube $Q_{j}=\phi_{j}(U_{j})$
are all defined on (the closure of) this Lipschitz set, thus (3.2)
applies. The deduction of the local estimate
\[
\left|v_{j}\right|_{r,\alpha}\leq C(r,\alpha,K,U)\left|h_{j}\right|_{r,\alpha}
\]
is a bit more subtle than the corresponding $C^{r}$ case (but still
simple), and involves a judicious application of differentiation under
the integral sign. Then, as in the $C^{r}$ case, there is a constant
\[
d=d(r,\alpha)=d\big(r,\mbox{max}_{i,j\leq m}\left|\phi_{ji}\right|_{r+1,\alpha}\big)\geq1
\]
permitting to estimate the global $C^{r,\alpha}$ norm of $v_{j}$
in terms of that on the cube times $d$. To get this constant, one
uses (3.2) together with the estimate for the norm of the composition
(still subject to conditions (i) and (ii) above, see e.g. \cite[p.369]{CDK};
here $g:A\rightarrow B=\text{dom}\,f$),
\[
\begin{array}{lll}
\left|f\circ g\right|_{r,\alpha;A} & \leq & C(r,A,B)\left|f\right|_{r,\alpha;B}\left(1+\left|g\right|_{r,\alpha;A}^{r+\alpha}\right)\\
 & \leq & C(r,A,B)\left|f\right|_{r,\alpha;B}\left(1+\text{max}\left(\left|g\right|_{r,\alpha;A}^{r},\left|g\right|_{r,\alpha;A}^{r+1}\right)\right)
\end{array}
\]

Finally, we observe that Constant $C$ in Theorem \hyperlink{th3}{3}
actually does not depend on the Hölder exponent $\alpha$, as $C$
ultimately depends only on $r$ and on the $C^{r,\alpha}$ and $C^{r+1,\alpha}$
norms of finitely many smooth functions depending only on $K$ and
$U$ or even only on the atlas (this is the case for the chart transition
maps $\phi_{ji}$). On local charts, the domains $A$ of these functions
are always Lipschitz (see above), thus, for each such function, all
these norms (with $\alpha$ in the range $]0,1]$) are uniformly estimated
in terms of the respective $C^{r+2}$ norm times a constant $C(r,A)$
(\cite[p.342]{CDK}). Taking the maximum of these constants for the
finitely many functions involved, we get a constant $\widehat{C}=\widehat{C}(r,K,U)$,
enabling the simultaneous estimate of all these $C^{r,\alpha}$ and
$C^{r+1,\alpha}$ norms ($0<\alpha\leq1)$ in terms of the respective\textbf{
}$C^{r+2}$ norms times $\widehat{C}$. Thus $C$ depends only on
$r$, $K$ and $U$. 

\hypertarget{se4}{}

\section{Linearized conservative Franks lemma}

We now state the linearized volume preserving version of Franks lemma.
Since perturbations of diffeomorphisms are usually carried out via
chart representations, as with Theorem \hyperlink{th6}{6}, care has
been taken to find a $\delta$ that directly estimates the permitted
variation of the derivative on all chart representations (see Section
\hyperlink{se5.2.3}{5.2.3} for the terminology). We start by stating
a simpler topological version of this result. The full strength is
achieved in Theorem \hyperlink{th}{8}.

\hypertarget{th7}{}
\begin{thm}
\emph{(Linearized conservative Franks lemma). }Let $M$ be a manifold
as in Section \hyperlink{se2}{2}. Fix $r\in\mathbb{Z}^{+}$ and $0<\alpha<1$
and let $\mathcal{\mathcal{U}}$ be a $C^{1}$ neighbourhood of $f\in\text{\emph{Diff}}_{\mu}^{\,r,\alpha}(M)$
in $\text{\emph{Diff}}_{\mu}^{\,r,\alpha}(M)$. Then, there is a smaller
$C^{1}$ neighbourhood $\mathcal{U}_{0}$ of $f$ in $\text{\emph{Diff}}_{\mu}^{\,r,\alpha}(M)$
and $\delta=\delta(r,\alpha,f,\mathcal{U})>0$ such that: given

- any $g\in\mathcal{U}_{0}$;

- any finite set $\varSigma\subset M$;

- any neighbourhood $U$ of $\varSigma$;

- any linear maps $A_{x}\in SL(n,\mathbb{R})$, $x\in\varSigma$,
satisfying
\[
\left\Vert A_{x}-Dg_{x}(x)\right\Vert <\delta
\]
where $g_{x}$ is some chart representation of $g$ around $x$, there
exists $\widetilde{g}\in\mathcal{U}$ having, for each $x\in\varSigma$,
a chart representation $\widetilde{g_{x}}$ around $x$ comparable
with $g_{x}$ and such that:
\begin{enumerate}
\item $\widetilde{g_{x}}(y)=g_{x}(x)+A_{x}(y-x)$ near $x$;
\item $\text{\emph{supp}}(\widetilde{g}-g)\subset U$.
\end{enumerate}
\noindent Furthermore, if g is $C^{\infty}$ then so is $\widetilde{g}$. 
\end{thm}
\hypertarget{r5}{}
\begin{rem}
for each $x\in\varSigma$, (1) implies $\widetilde{g}(x)=g(x)$, $D\widetilde{g_{x}}(x)=A_{x}$
and $\widetilde{g}$ is affine linear near $x$ in chart representation
$\widetilde{g_{x}}$.
\end{rem}
The proof actually establishes the stronger result stated below. Given
a $C^{1}$ diffeomorphism $f$ of $M$ onto itself, let $\text{sup}_{M}\|Df\|$
denote the supremum of $\|Df(y)\|$ for all $y\in M$, over all possible
chart representations of $f$ around $y$ (see Section \hyperlink{se5.2}{5.2}).
As in chart representations the derivatives of a conservative diffeomorphism
belong to $SL(n,\mathbb{R})$, imposing an uniform upper bound $\text{sup}_{M}\|Df\|\leq d$
automatically guarantees uniform local bounded distortion for all
conservative diffeomorphisms satisfying this inequality: on chart
representations, for any $x\in M$, the image of $\mathbb{S}^{n-1}$
under the derivative $Df(x;\cdot)$ is an ellipsoid with major radius
$\leq d$ and minor radius $\geq d^{-n+1}$ (this is immediate looking
at the polar decomposition).

Also, as it is shown below in part (C) of the proof of Lemma \hyperlink{l2}{2},
$\delta$ can be made to depend linearly on the required $C^{1}$-closeness
$\epsilon$ of the resulting diffeomorphism $\widetilde{g}$ to $g$
(provided $\epsilon$ is small enough). With both observations in
mind, Theorem \hyperlink{th7}{7} can be reformulated as follows:

\hypertarget{th8}{}
\begin{thm}
\emph{(Linearized conservative Franks lemma). }Let $M$ be a manifold
as in Section \hyperlink{se2}{2}. Fix $r\in\mathbb{Z}^{+}$, $0<\alpha<1$
and $d\geq1$. Then, there is a constant $\chi=\chi(r,\alpha,d)>0$
such that: given 

- any $g\in\text{\emph{Diff}}_{\mu}^{\,r,\alpha}(M)$ with $\text{\emph{sup}}_{M}\|Dg\|\leq d$ 

- any $0<\epsilon\leq1$;

- any finite set $\varSigma\subset M$;

- any neighbourhood $U$ of $\varSigma$;

- any linear maps $A_{x}\in SL(n,\mathbb{R})$, $x\in\varSigma$,
satisfying
\[
\left\Vert A_{x}-Dg_{x}(x)\right\Vert <\chi\epsilon
\]
where $g_{x}$ is some chart representation of $g$ around $x$, then
(adopting any local $C^{1}$-metrization of $\text{\emph{Diff}}_{\mu}^{\,r,\alpha}(M)$
near $g$ as in Section \hyperlink{se5.2}{5.2}), there exists $\widetilde{g}\in\text{\emph{Diff}}_{\mu}^{\,r,\alpha}(M)$
$\epsilon\text{-}C^{1}$-close to $g$ having, for each $x\in\varSigma$,
a chart representation $\widetilde{g_{x}}$ around $x$ comparable
with $g_{x}$ and such that:
\begin{enumerate}
\item $\widetilde{g_{x}}(y)=g_{x}(x)+A_{x}(y-x)$ near $x$;
\item $\text{\emph{supp}}(\widetilde{g}-g)\subset U$.
\end{enumerate}
\noindent Furthermore, if g is $C^{\infty}$ then so is $\widetilde{g}$. 
\end{thm}
\hypertarget{r6}{}
\begin{rem}
Avila's localized smoothing \cite[Theorem 7]{AV} implies that Theorem
\hyperlink{th7}{7} can be stated for $\text{Diff}_{\mu}^{1}(M)$
in place of $\text{Diff}_{\mu}^{\,r,\alpha}(M)$ (with $\chi=\chi(d)>0$),
the reduction of the $C^{1}$ local linearization to the $C^{\infty}$
case being then achieved through Lemma \hyperlink{l2}{2} below. However,
if $g$ is $C^{k}$, $k\geq2$ an integer, one should not be tempted
to apply \cite[Theorem 7]{AV} in order to smooth $g$ near $x$ (getting
$\widehat{g}$), then apply the elementary perturbation lemma\textbf{
\cite[Lemma A.4, p.93]{BC}} to correct $\widehat{g}(x)$ back to
$g(x)$ and finally apply Lemma \hyperlink{l2}{2} below to get a
$C^{1}$ perturbation $\widetilde{g}$, still of class $C^{k}$, which
is affine linearized near $x$ (in some chart representation) and
coincides with $g$ at $x$ and outside any given small neighbourhood
of this point. Indeed, \cite{AV} does not guarantee the resulting
map to be $C^{2}$ at the boundary points of the open set $\varOmega$
where the smoothing takes place, the above reasoning being valid only
for $k=1$.
\end{rem}
Obviously, the $C^{1}$-closeness of $\widetilde{g}$ to $g$ is the
best possible and cannot be upgraded to any of the higher $C^{1+}$
topologies (even if the localized support is dropped and $\varSigma$
is reduced to a single point). In terms of regularity, Theorems \hyperlink{th7}{7}
and \hyperlink{th8}{8} are also optimal, in the sense that the resulting
diffeomorphism $\widetilde{g}$ is still $C^{r,\alpha}$ (respect.
$C^{\infty}$) as the original one. If one is particularly interested
in the class of $C^{k}$ diffeomorphisms, $k\geq2$ an integer, it
is natural to ask if $\widetilde{g}$ can be found of class $C^{k}$
as $g$ and not merely of class $C^{k-1,\alpha}$ for any chosen $0<\alpha<1$
(a version of this statement appears without proof in \cite[p.217]{HHTU}).
For $k\geq2$, a positive answer seems beyond the techniques presently
available (if possible at all). The case $k=1$ is exceptional due
to Avila's theorem mentioned above, but no analogue result is known
for $k\geq2$. These difficulties are related to the fact that, in
dimension $n\geq2$, there are, in general, no known $C^{r+1}$ solutions
to the prescribed Jacobian PDE, $\text{det\,}Df=h$, when $h$ is
of class $C^{r}$, $r\in\mathbb{Z}^{+}$ (see e.g. \cite[p.192]{CDK},
\textbf{\cite[p.324]{RY}}). 

In virtue of Lemma \hyperlink{l2}{2} below, the answer would be positive
if $g$ could be $C^{k+}$-smoothened near $0$, i.e. if one could
answer affirmatively the following

\emph{\smallskip{}
}\noun{Question }(\emph{Local $C^{k+}$-smoothing with $C^{1}$-closeness}):
Given any volume preserving $C^{k}$ map $g:\mathbb{B}^{n}\longrightarrow\mathbb{R}^{n}$,
$k\geq2$ an integer, is there arbitrarily $C^{1}$-close to it another
volume preserving $C^{k}$ map $\widehat{g}:\mathbb{B}^{n}\longrightarrow\mathbb{R}^{n}$
which is $C^{k,\alpha}$ near $0$ (for some $0<\alpha<1)$ and satisfies
$\text{supp}(\widehat{g}-g)\subset\subset\mathbb{B}^{n}\,$? 
\begin{proof}
(Theorem \hyperlink{th7}{7}). We shall reduce the proof to that of
Lemma \hyperlink{l2}{2} below. Fix a covering system $\{B_{l}\}_{l\leq\widetilde{m}}$,
$i,j$ for $f$, here called $\varUpsilon$, as in Section \hyperlink{se5.2}{5.2}
and $0<\epsilon\leq1$ such that $\mathcal{\mathscr{U}}_{\epsilon,\varUpsilon}(f)\subset\mathcal{U}$.
Let $\mathcal{U}_{0}=\mathcal{\mathscr{U}}_{\epsilon/2,\varUpsilon}(f)$.
Recall that, by definition of $\mathcal{\mathscr{U}}_{\epsilon,\varUpsilon}(f)$,
the same covering system works for any $g\in\mathcal{\mathscr{U}}_{\epsilon,\varUpsilon}(f)$.
Let $g\in\mathcal{U}_{0}$. As one wishes, for each $x\in\varSigma$,
to be able to choose freely any chart representation $g_{x}$ around
$x$ where to perform the local linearization (getting $\widetilde{g_{x}}$),
we will need to estimate $\text{sup}_{M}\|Dg\|$ for all such $g$,
the supremum of $\|Dg(y)\|$ for all $y\in M$, over all possible
chart representations of $g$ around $y$ (see Section \hyperlink{se5.2}{5.2}).
The transitions between chart representations of $g$ around point
$x$ being of the form $g_{\widehat{j}\widehat{i},B}=\phi_{\widehat{j}j}\circ g_{ji;B}\circ\phi_{i\widehat{i}}$
(Section \hyperlink{se}{5.2.2}), one gets, as $\epsilon\leq1$, for
all $g\in\mathcal{U}_{0}$,
\[
\text{sup}_{M}\|Dg\|<c:=a^{2}(\text{sup}_{M}\|Df\|+1)
\]
where
\[
a=\underset{i,j\leq m}{\text{max}}\underset{\phi_{i}(V_{i}\cap V_{j})}{\text{sup}}\left\Vert D\phi_{ji}\right\Vert 
\]
$\phi_{ji}=\phi_{j}\circ\phi_{i}^{-1}$ being the chart transitions
of the atlas $(V_{i},\,\phi_{i})_{i\leq m}$. Note that we need not
concern with the $C^{0}$ norm of $\widetilde{g}-g$ since it becomes
as small as wished if $\text{supp}(\widetilde{g}-g$) is contained
in the disjoint union of sufficiently small open balls (on local charts)
centred at the points of $\varSigma$. This also guarantees that (2)
holds. Hence, only the distance between the derivatives of $\widetilde{g}$
and $g$ is of concern. Performing adequate translations in both domain
and target of each chart representation $g_{x}$ around $x$, it is
now easily seen that that the problem reduces to prove Lemma \hyperlink{l2}{2}
below and finding through it the constant $\chi=\chi(r,\alpha,c,n)$
and then let $\delta=\chi\epsilon_{0}$ where $\epsilon_{0}=\epsilon/2b$.
Here, $b\geq1$ is a multiplicative constant (to be determined below)
controlling the possible magnification of the distance $\text{\,}\left\Vert D\widetilde{g_{x}}(y)-Dg_{x}(y)\right\Vert $,
$y\in\text{supp}(g_{x}-\widetilde{g_{x}})$, when passing from $g_{x},\,\widetilde{g_{x}}$
to any other pair $\widehat{g},\,\widehat{\widetilde{g}}$ of comparable
chart representations of $g$ and $\widetilde{g}$ around $y$. This
will guarantee, in particular, that for $g\in\mathcal{U}_{0},$ $\left\Vert \widetilde{g}-g\right\Vert _{C^{1}}<\epsilon/2$
in the local metric induced on $\mathcal{\mathscr{U}}_{\epsilon,\varUpsilon}(f)$,
and therefore that one gets as wished
\[
\left\Vert \widetilde{g}-f\right\Vert _{C^{1}}\leq\left\Vert \widetilde{g}-g\right\Vert _{C^{1}}+\left\Vert g-f\right\Vert _{C^{1}}<\epsilon/2+\epsilon/2=\epsilon
\]
We now construct $\widetilde{g}$ and proceed to determine the constant
$b$ mentioned above. Since this is more subtle than it might seem
at first sight we do it with some detail. To simplify the exposition,
we identify a point $x$ in $M$ with its image $\phi_{i}(x)$ in
a local chart. We first select at will, for each $x\in\varSigma$,
a chart representation $g_{x}=g_{ji,D}$ of $g$ around $x$ and fix
a small closed ball $B_{x}$ centred at this point and contained in
the (open) domain $\phi_{i}(D$) of $g_{x}$, such that the $B_{x}$'s
are mutually disjoint (i.e. have mutually disjoint preimages in $M$).
Using Lemma \hyperlink{l2}{2} below, we find $\delta=\chi\epsilon_{0}$
where $\epsilon_{0}=\epsilon/2b$ and 
\[
b=n^{2}a\Big(a+(c+1)\underset{i,j\leq m}{\text{max}}\underset{\phi_{i}(V_{i}\cap V_{j})}{\text{sup}}\big(\|D^{2}\phi_{ji}\|+1\big)\Big)
\]
and then, for any given $A_{x}\in SL(n,\mathbb{R})$ as in the statement
of Theorem \hyperlink{th7}{7}, we find a volume preserving $C^{r,\alpha}$
(respect. $C^{\infty}$) diffeomorphism onto its image $\widetilde{g_{x}}:\phi_{i}(D)\rightarrow\phi_{j}(V_{j})$,
which is affine linearized by $A_{x}$ near $x$, and satisfies $\widetilde{g_{x}}(x)=g_{x}(x)$,
$\text{supp}(\widetilde{g_{x}}-g_{x})\subset B_{x}$ and $\left\Vert \widetilde{g_{x}}-g_{x}\right\Vert {}_{C^{1}}<\epsilon/2b$.
In this way we have $\widetilde{g}$ globally defined: $\widetilde{g}=\widetilde{g_{x}}$
in $B_{x}$ and $\widetilde{g}=g$ in $(\cup_{x\in\varSigma}B_{x})^{c}$
(again, we simplify the notation identifying $\widetilde{g_{x}}$
with the corresponding map in $M$ and $B_{x}$ with its preimage
in $M$). Now, let $\widehat{g}=\widehat{g}_{\widehat{j}\widehat{i},E}$
and $\,\widehat{\widetilde{g}}=\,\widehat{\widetilde{g}}{}_{\widehat{j}\widehat{i},E}$
be any other pair of comparable chart representations of $g$ and
$\widetilde{g}$ around the preimage $\widehat{y}=\phi_{i}^{-1}(y)$
in $M$ of $y\in\text{supp}(\widetilde{g_{x}}-g_{x})$. We claim that
\begin{equation}
\|D\widehat{\widetilde{g}}(\widehat{y})-D\widehat{g}(\widehat{y})\|<\epsilon/2
\end{equation}
as wished. From the expression giving the derivative under chart representation
transition,
\[
D\widehat{g}(\widehat{y})=D\phi_{\widehat{j}j}(g_{x}(y))\circ Dg_{x}(y)\circ D\phi_{i\widehat{i}}(\widehat{y}),\quad y=\phi_{i\widehat{i}}(\widehat{y})
\]
one gets that
\[
\|D\widehat{\widetilde{g}}(\widehat{y})-D\widehat{g}(\widehat{y})\|
\]
is less or equal than
\begin{equation}
\|D\phi_{\widehat{j}j}(\widetilde{g_{x}}(y))\circ D\widetilde{g_{x}}(y)-D\phi_{\widehat{j}j}(g_{x}(y))\circ Dg_{x}(y)\|\cdot\|D\phi_{i\widehat{i}}(\widehat{y})\|
\end{equation}

\noindent (i) If $\widetilde{g_{x}}(y)=g_{x}(y)$, then the norm on
the left equals
\[
\|D\phi_{\widehat{j}j}(g_{x}(y))\|\cdot\|D\widetilde{g_{x}}(y)-Dg_{x}(y)\|
\]
hence
\[
\|D\widehat{\widetilde{g}}(\widehat{y})-D\widehat{g}(\widehat{y})\|\;\leq a^{2}\left\Vert \widetilde{g_{x}}-g_{x}\right\Vert {}_{C^{1}}<a^{2}\epsilon/2b<\epsilon/2
\]

\noindent (ii) If $\widetilde{g_{x}}(y)\neq g_{x}(y)$, then denoting
by $\mathcal{M}(y)=[a_{kl}]$ the $n\times n$ matrix in (4.2) inside
the norm on the left, we have for the constant $a$ defined above,

\noindent 
\[
\|D\widehat{\widetilde{g}}(\widehat{y})-D\widehat{g}(\widehat{y})\|\leq a\left\Vert \mathcal{M}(y)\right\Vert 
\]
We estimate the absolute value of the entries $a_{kl}$ and then use
$\left\Vert \mathcal{M}(y)\right\Vert \leq n\,\text{max}|a_{kl}|$.
Denoting by $\phi^{k}$ the $k$-th component of $\phi_{\widehat{j}j}$
and by $\{e_{i}\}_{i\leq n}$ the canonical base of $\mathbb{R}^{n}$,
\[
|a_{kl}|=\left|\sum_{i=1}^{n}\partial_{e_{i}}\phi^{k}(\widetilde{g_{x}}(y))\cdot\partial_{e_{l}}\widetilde{g_{x}}^{i}(y)-\partial_{e_{i}}\phi^{k}(g_{x}(y))\cdot\partial_{e_{l}}g_{x}^{i}(y)\right|
\]
Now, the key step is to write (using the mean value theorem),
\begin{equation}
\partial_{e_{i}}\phi^{k}(\widetilde{g_{x}}(y))=\partial_{e_{i}}\phi^{k}(g_{x}(y))+\partial_{u}\partial_{e_{i}}\phi^{k}(z)\cdot|\widetilde{g_{x}}(y)-g_{x}(y)|
\end{equation}
where $z$ is some point in the interior of segment $[\widetilde{g_{x}}(y),\,g_{x}(y)]$
and $u$ is the direction $\frac{\widetilde{g_{x}}(y)-g_{x}(y)}{|\widetilde{g_{x}}(y)-g_{x}(y)|}$.
Since 
\[
|\partial_{e_{l}}\widetilde{g_{x}}^{i}(y)|\leq\left\Vert \widetilde{g_{x}}\right\Vert _{C^{1}}<\left\Vert g_{x}\right\Vert _{C^{1}}+\epsilon/2b<c+1
\]
a simple calculation shows that
\begin{equation}
|a_{kl}|\leq n\Big(a+(c+1)\underset{i,j\leq m}{\text{max}}\underset{\phi_{i}(V_{i}\cap V_{j})}{\text{sup}}\left\Vert D^{2}\phi_{ji}\right\Vert \Big)\left\Vert \widetilde{g_{x}}-g_{x}\right\Vert _{C^{1}}
\end{equation}
and since $\left\Vert \widetilde{g_{x}}-g_{x}\right\Vert {}_{C^{1}}<\epsilon/2b$,
inequality (4.1) follows. The problem with the above reasoning is
that the segment $[\widetilde{g_{x}}(y),\,g_{x}(y)]$ might \emph{not
}be contained in the domain $\phi_{j}(V_{j}\cap V_{\widehat{j}})$
of $\phi_{\widehat{j}j}$ and reducing $\text{supp}(\widetilde{g_{x}}-g_{x})$
to an even smaller neighbourhood of $x$ will not help if $g(x)\in\overline{V_{\widehat{j}}}\setminus V_{\widehat{j}}$.
To overcome this difficulty we use the fact that the atlas of $M$
is contained in a larger atlas (see the Convention, Section \hyperlink{se2}{2}):
there is a small $\varrho>0$ such that, for every chart domain $V_{k}$,
$\text{sup}\|D^{2}\varPhi_{kj}\|$ evaluated in the $\varrho$-neighbourhood
$\varDelta_{kj}$ of $\phi_{j}(V_{k}\cap V_{j}$) is smaller than
$\text{sup}\|D^{2}\phi_{kj}\|+1$ in $\phi_{j}(V_{k}\cap V_{j})$.
If necessary, we then reduce the radius of the closed ball $B_{x}$
even further so that $g_{x}(B_{x})\subset\phi_{j}(V_{j})$ has diameter
smaller than $\varrho$. As $y\in\text{supp}(\widetilde{g_{x}}-g_{x})\subset B_{x}$,
both $g_{x}(y)$ and $\widetilde{g_{x}}(y)$ are contained in $g_{x}(B_{x})$,
thus the segment $[\widetilde{g_{x}}(y),\,g_{x}(y)]$ is entirely
contained in $\varDelta_{\widehat{j}j}$. It is thus enough to replace
$\phi=\phi_{\widehat{j}j}$ in (4.3) by its extension $\varPhi_{\widehat{j}j}$
and replace $\|D^{2}\phi_{ji}\|$ by $\|D^{2}\phi_{ji}\|+1$ in (4.4),
as it is done in the definition of $b$. We have thus reduced the
proof of Theorem \hyperlink{th7}{7} to that of Lemma \hyperlink{l2}{2}
below. 
\end{proof}
From now on we assume that $\mathbb{\mathbb{R}}^{n}$ and all its
subsets are endowed with the standard volume form $dx_{1}\wedge\ldots\wedge dx_{n}$.
We write $A\subset\subset B$ for ``$A$ is compact and contained
in $B$''.

\hypertarget{l2}{}
\begin{lem}
\emph{(Uniform conservative local linearization) }Given any $n,\,r\in\mathbb{Z}^{+}$,
$0<\alpha<1$ and $c\geq1$ there exists a constant $\chi=\chi(r,\alpha,c,n)>0$
such that: given any 

\emph{(a)} $0<\epsilon_{0}\leq1$;

\emph{(b)} any volume preserving $C^{r,\alpha}$ diffeomorphism onto
its image\emph{
\[
f:\eta\mathbb{B}^{n}\longrightarrow\mathbb{R}^{n},\,\,\,\eta>0
\]
}such that $f(0)=0$ and 
\[
\left\Vert Df(0)\right\Vert \leq c
\]

\emph{(c)} any $A\in SL(n,\mathbb{R})$ such that
\[
\left\Vert A-Df(0)\right\Vert <\chi\epsilon_{0}
\]
there exists a volume preserving $C^{r,\alpha}$ diffeomorphism onto
its image\emph{ $f_{A}:\eta\mathbb{B}^{n}\rightarrow\mathbb{R}^{n}$
}satisfying
\begin{enumerate}
\item $f_{A}=A$ near $0$;
\item $\text{\emph{supp}}(f_{A}-f)\subset\subset\eta\mathbb{B}^{n}$;
\item $\left\Vert f_{A}-f\right\Vert _{C^{1}}<\epsilon_{0}$.
\end{enumerate}
Furthermore, if $f$ is $C^{\infty}$ then so is $f_{A}$. 
\end{lem}
\begin{proof}
We treat the cases (\hyperlink{(A)}{A}) $f\in C^{r,\alpha}\setminus C^{\infty}$
and (\hyperlink{(B)}{B}) $f\in C^{\infty}$ separately. In order
to make the construction of $f_{A}$ more transparent, we start by
establishing in Case (A), through a continuity reasoning, the existence
for each $\epsilon_{0}>0$ of a $\delta=\delta(r,\alpha,c,n,\epsilon_{0})>0$
such that (1) - (3) hold if $\|A-Df(0)\|<\delta$, and analogously
$\delta=\delta(c,n,\epsilon_{0})>0$ is found in Case (B). Finally,
the linear dependence of $\delta$ on $\epsilon_{0}$ for $0<\epsilon_{0}\leq1$
is established in each case (see (C) and (D) below), getting $\delta=\chi\epsilon_{0}$
for some constant $\chi=\chi(r,\alpha,c,n)>0$ in case (A) and for
$\chi=\chi(c,n)>0$ in Case (B). We then take $\chi$ as the minimum
of these two values. \smallskip{}

\hypertarget{(A)}{}

\noindent \textbf{(A).} \noun{Case}\emph{ }$f\in C^{r,\alpha}\setminus C^{\infty}$.
\smallskip{}

The following auxiliary fact follows readily from the compactness
of 
\[
SL_{c}:=\{D\in SL(n,\mathbb{R}):\,\left\Vert D\right\Vert \leq c\},\,\,c\geq1
\]
and the continuity of the the composition operator for matrices in
relation to the standard norm. Together with Fact \hyperlink{F2}{2}
below, it will ultimately permit to find,\textbf{ }for given $\epsilon_{0}>0$
and $c\geq1$, a single $\delta$ working simultaneously for all $f$
satisfying (b). Proofs of both Facts with linear estimates are given
in (C).

\hypertarget{F1}{}
\begin{fact}
For any $n\in\mathbb{Z}^{+}$, $\epsilon>0$ and $c\geq1$ there is
$\delta>0$ such that: given any $A\in SL(n,\mathbb{R})$ and $D\in SL_{c}$
\[
\left\Vert A-D\right\Vert <\delta\Longrightarrow\left\Vert A^{-1}\circ D-\mathcal{\text{\emph{Id}}}\right\Vert <\epsilon
\]
\end{fact}
The precise $\epsilon_{0}-\delta$ chain establishing Lemma \hyperlink{l2}{2}
can be easily reconstructed from the following reasoning, which makes
the structure of the proof more transparent. The continuity of the
addition and multiplication operators in relation to the $C^{r,\alpha}$
norm and that of the composition and inversion operators in relation
to the $C^{1}$ norm will be systematically used without mention. 

While Lemma \hyperlink{l2}{2} is a $C^{1}$-closeness result, we
will need to work with the $C^{1,\alpha}$ norm until step (A.2) in
order to guarantees that the volume correcting diffeomorphism $\varphi^{-1}$
is of class $C^{r,\alpha}$ and $C^{1}$-close to $\text{Id}$. Then
we return to the standard Whitney $C^{1}$ norm using $\left\Vert \cdot\right\Vert _{C^{1}}\leq n\left|\cdot\right|_{1}\leq n\left|\cdot\right|_{1,\alpha}$
(see Section \hyperlink{se5.1}{5.1}). 

For $h\in C^{r,\alpha}(\overline{\mathbb{B}^{n}},\mathbb{R}^{n})$,
$r\in\mathbb{Z}^{+}$, $0<\alpha\leq1$, we adopt the $C^{r,\alpha}$
norm corresponding to that of Section \hyperlink{se3.2}{3.2} (for
$h\in C^{r,\alpha}(\overline{\mathbb{B}^{n}})$ the definition is
the same but the component superscript $i$ disappears). This is equivalent
to the standard Whitney-Hölder $C^{r+\alpha}$ norm $\left\Vert \cdot\right\Vert _{C^{r,\alpha}}$.
\[
\left|h\right|_{r,\alpha;\mathbb{B}^{n}}=\underset{\substack{i;\,|\sigma|=r}
}{\mbox{max}}\left(\left|h\right|_{r;\mathbb{B}^{n}},\,\left[\partial^{\sigma}h^{i}\right]_{\alpha;\mathbb{B}^{n}}\right)
\]

\smallskip{}

\hypertarget{(A.1)}{}

\noindent (\textbf{A.1}) \emph{\noun{Reducing to the case of diffeomorphisms
with domain}}\emph{ $\mathbb{B}^{n}$ $C^{1,\alpha}$}\emph{\noun{-close
to}}\emph{ }$\text{Id}$ \emph{\noun{and}} $A=\text{Id}$. Let $0<\lambda<\text{min}(1,\eta)$.
For each $f$ of class $C^{r,\alpha}$ satisfying (b), re-scale $f|_{\lambda\mathbb{B}^{n}}$
to the unit ball getting a volume preserving $C^{r,\alpha}$ diffeomorphism
onto its image
\[
\begin{array}{llll}
f_{\lambda}: & \mathbb{B}^{n} & \longrightarrow & \mathbb{R}^{n}\\
 & z & \longmapsto & \lambda^{-1}f(\lambda z)
\end{array}
\]
One has,
\begin{equation}
\left|f_{\lambda}-Df(0)\right|_{1,\alpha;\mathbb{B}^{n}}\xrightarrow[\lambda\rightarrow0]{}0
\end{equation}
(Up to the $C^{1}$ norm, the reasoning is the same as in the proof
of Theorem \hyperlink{th6}{6}. Let $\{e_{j}\}_{j\leq n}$ be the
canonical base of $\mathbb{R}^{n}$. Writing $\partial_{j}$ for $\partial_{e_{j}}$,
one has for the partial derivatives of the components $f_{\lambda}^{i}$
of $f_{\lambda}$,
\[
\begin{array}{lll}
\\
\underset{x,y\in\mathbb{B}^{n};x\neq y}{\mbox{sup}}\frac{\left|\partial_{j}f_{\lambda}^{i}(y)-\partial_{j}f_{\lambda}^{i}(x)\right|}{\left|y-x\right|^{\alpha}} & = & \underset{x,y\in\mathbb{B}^{n};x\neq y}{\mbox{sup}}\lambda^{\alpha}\frac{\left|\partial_{j}f^{i}(\lambda y)-\partial_{j}f^{i}(\lambda x)\right|}{\left|\lambda y-\lambda x\right|^{\alpha}}\\
 & \,\\
 & \leq & \lambda^{\alpha}\left|f\right|_{1,\alpha;\lambda\mathbb{B}^{n}}\xrightarrow[\lambda\rightarrow0]{}0
\end{array}
\]
thus establishing (4.5)). For each $A\in SL(n,\mathbb{R})$ let 
\[
h_{A,\lambda}=A^{-1}\circ f{}_{\lambda}
\]
By (4.5) (see e.g. \cite[p.384]{CDK}), 
\begin{equation}
\left|h_{A,\lambda}-A^{-1}\circ Df(0)\right|_{1,\alpha;\mathbb{B}^{n}}\xrightarrow[\lambda\rightarrow0]{}0
\end{equation}
Fix $\xi\in C^{\infty}(\mathbb{B}^{n};[0,1])$ (the same for all $f$
and $A$) with $\xi=0$ in $\frac{1}{3}\mathbb{D}^{n}$ and $\xi=1$
in $\mathbb{B}^{n}\setminus\frac{2}{3}\mathbb{B}^{n}$ and define
\[
g_{A,\lambda}=\text{Id}+\xi(h_{A,\lambda}-\text{Id})
\]
Then, noting that for $L\in L(n,\mathbb{R})$, $\left|L\right|_{1,\alpha;\mathbb{B}^{n}}\leq\|L\|$,
by (4.6)
\begin{equation}
\left|h_{A,\lambda}-\text{Id}\right|_{1,\alpha;\mathbb{B}^{n}}\xrightarrow[\lambda\rightarrow0]{}\left|A^{-1}\circ Df(0)-\text{Id}\right|_{1,\alpha;\mathbb{B}^{n}}\leq\left\Vert A^{-1}\circ Df(0)-\text{Id}\right\Vert 
\end{equation}
and by Fact \hyperlink{F1}{1} above, as $Df(0)\in SL_{c}$, for $\delta$
small the norm on the right is uniformly small for all $f$ satisfying
(b) and all $A$ satisfying (c), hence for $\lambda$ small enough
\begin{equation}
\left|g_{A,\lambda}-\text{Id}\right|_{1,\alpha;\mathbb{B}^{n}}\text{ is small}
\end{equation}
and, in particular, $g_{A,\lambda}$ is a diffeomorphism of $\mathbb{B}^{n}$
onto its image.

\smallskip{}

\hypertarget{(A.2)}{}

\noindent (\textbf{A.2}) \noun{Correcting the volume distortion.}
Dropping the subscripts for simplicity, let
\[
\theta=\theta_{A,\lambda}=\text{det}\,Dg_{A,\lambda}
\]
Then by (4.8), 
\begin{equation}
\text{(i)}\quad\left|\theta-\text{1}\right|_{0,\alpha;\mathbb{B}^{n}}\text{ is small}
\end{equation}
and (ii)$\int_{\mathbb{B}^{n}}\theta=\text{meas\,}g_{A,\lambda}(\mathbb{B}^{n})=\text{meas\,}\mathbb{B}^{n}$
and (iii) $\theta=1$ in $\mathscr{C}=\frac{1}{3}\mathbb{D}^{n}\,\cup\,(\mathbb{B}^{n}\setminus\frac{2}{3}\mathbb{B}^{n})$.
Now, it is easily seen that we can apply \cite[Theorem 4]{TE} (with
$\gamma=\alpha$), to get $\varphi\in\text{Diff}^{r,\alpha}(\mathbb{B}^{n})$
such that $\text{det\,}D\varphi=\theta$ and $\varphi=\text{Id in \ensuremath{\mathscr{C}}}$,
with
\begin{equation}
\left|\varphi-\text{Id}\right|_{1;\mathbb{B}^{n}}\text{ small}
\end{equation}
Then,
\[
\widetilde{g_{A,\lambda}}=g_{A,\lambda}\circ\varphi^{-1}
\]
is a volume preserving $C^{r,\alpha}$ diffeomorphism of $\mathbb{B}^{n}$
onto its image with 
\[
\begin{cases}
\widetilde{g_{A,\lambda}}=\text{Id} & \text{in\, }{\textstyle \frac{1}{3}}\mathbb{D}^{n}\\
\widetilde{g_{A,\lambda}}=h_{A,\lambda} & \text{in\, }{\textstyle \mathbb{B}^{n}\setminus\frac{2}{3}}\mathbb{B}^{n}
\end{cases}
\]
and
\begin{equation}
\left\Vert \widetilde{g_{A,\lambda}}-\text{Id}\right\Vert _{C^{1};\mathbb{B}^{n}}\text{ \,is small}
\end{equation}

\smallskip{}

\hypertarget{(A.3)}{}

\noindent (\textbf{A.3})\textbf{ }\noun{Back to the general case.}
Setting 
\[
\widetilde{f_{A,\lambda}}=A\circ\widetilde{g_{A,\lambda}}
\]
it is immediate to verify that
\[
\begin{cases}
\widetilde{f_{A,\lambda}}=A\text{ \,near \,0}\\
\text{supp\,}(\widetilde{f_{A,\lambda}}-f_{\lambda})\subset\subset\mathbb{B}^{n}
\end{cases}
\]
By (4.11) and Fact \hyperlink{F2}{2} below, for $\delta$ (and $\lambda)$
small 
\begin{equation}
\left\Vert \widetilde{f_{A,\lambda}}-Df(0)\right\Vert _{C^{1};\mathbb{B}^{n}}<\epsilon_{0}/2
\end{equation}
for all $f$ satisfying (b) and all $A$ satisfying (c). For a proof
of Fact \hyperlink{F2}{2} with linear estimate see (C.8) below.

\hypertarget{F2}{}
\begin{fact}
For any $n\in\mathbb{Z}^{+}$, $\epsilon>0$ and $c\geq1$ there is
$\delta>0$ such that: given any $A\in L(n,\mathbb{R})$, $D\in SL_{c}$
and a $C^{1}$ map $g:\mathbb{B}^{n}\rightarrow\mathbb{R}^{n}$, 
\[
\left\Vert A-D\right\Vert ,\,\left\Vert g-\text{\emph{Id}}\right\Vert _{C^{1};\mathbb{B}^{n}}<\delta\Longrightarrow\left\Vert A\circ g-D\right\Vert _{C^{1};\mathbb{B}^{n}}<\epsilon
\]
\end{fact}
\hypertarget{(A.4)}{}

\noindent (\textbf{A.4}) \emph{\noun{Scaling down to the real scale.}}\emph{
}It remains to scale down $\widetilde{f_{A,\lambda}}$ back to the
real scale. Let
\[
\begin{array}{llll}
f_{A,\lambda}: & \lambda\mathbb{B}^{n} & \longrightarrow & \mathbb{R}^{n}\\
 & z & \longmapsto & \lambda\widetilde{f_{A,\lambda}}(\lambda^{-1}z)
\end{array}
\]
Since the $C^{1}$ norm does not increase under contracting homothetic
conjugation and $\lambda<1$, 
\begin{equation}
\left\Vert f_{A,\lambda}-Df(0)\right\Vert _{C^{1};\lambda\mathbb{B}^{n}}\leq\left\Vert \widetilde{f_{A,\lambda}}-Df(0)\right\Vert _{C^{1};\mathbb{B}^{n}}<\epsilon_{0}/2
\end{equation}
Taking $\lambda$ even smaller if necessary, we can further guarantee
that
\[
\left\Vert f-Df(0)\right\Vert _{C^{1};\lambda\mathbb{B}^{n}}<\epsilon_{0}/2
\]
Therefore, as supp$\,(f_{A,\lambda}-f|_{\lambda\mathbb{B}^{n}})\subset\subset\lambda\mathbb{B}^{n}$,
extending $f_{A}:=f_{A,\lambda}$ by $f$ to the whole $\mathbb{\eta B}^{n}$
we finally get by the triangle inequality that
\[
(3)\quad\left\Vert f_{A}-f\right\Vert _{C^{1};\mathbb{\eta B}^{n}}<\epsilon_{0}
\]
and it is immediate to check that $f_{A}$ is $C^{r,\alpha}$ and
satisfies all the conclusions of Lemma \hyperlink{l2}{2}.\smallskip{}

\hypertarget{(B)}{}

\noindent \textbf{\emph{\noun{(B). }}}\noun{Case}\emph{ $f\in C^{\infty}$}.
Fixed $n$, $c$ and $\epsilon_{0}$, both the determination of $\delta=\delta(c,n,\epsilon_{0})$
and the construction of $f_{A}$ are similar to those in case (\hyperlink{(A)}{A}),
except that the volume correcting diffeomorphism $\varphi$ in (\hyperlink{(A.2)}{A.2})
must be obtained by a different method, as using \cite[Theorem 4]{TE},
there is no guarantee that the solution to $\text{det}\,D\varphi=\theta$
is smooth when $\theta$ is smooth (in the later case, we get a solution
$\varphi_{r}$ of class $C^{r}$, for each $r\in\mathbb{Z}^{+}$,
but a priori nothing guarantees that these $\varphi_{r}$ coincide
to form a $\,\text{\ensuremath{C^{\infty}} }$diffeomorphism. Reciprocally,
\cite[Theorem 5]{TE} and \cite[Lemma 10.4]{CDK} employed below cannot
be applied in case (\hyperlink{(A)}{A}) since it does not provide
the necessary gain of regularity, from $C^{r-1,\alpha}$(determinant
$\theta$) to $C^{r,\alpha}$ (diffeomorphism $\varphi$)). Here all
functions involved are smooth and 
\[
\left|f_{\lambda}-Df(0)\right|_{2;\mathbb{B}^{n}}\xrightarrow[\lambda\rightarrow0]{}0
\]
(up to the $C^{1}$ norm see (A.1) above; for each multiindex $\sigma$
of order 2, one has for the partial derivatives of the components
of $f_{\lambda}$,
\[
\underset{x\in\mathbb{B}^{n}}{\mbox{sup}}\left|\partial^{\sigma}f_{\lambda}^{i}(x)\right|=\underset{x\in\mathbb{B}^{n}}{\text{sup}}\,\lambda\left|\partial^{\sigma}f^{i}(\lambda x)\right|\leq\lambda\left|f\right|_{2;\lambda\mathbb{B}^{n}}\xrightarrow[\lambda\rightarrow0]{}0
\]
Therefore,
\begin{equation}
\left|h_{A,\lambda}-A^{-1}\circ Df(0)\right|_{2;\mathbb{B}^{n}}\xrightarrow[\lambda\rightarrow0]{}0
\end{equation}
and consequently, reasoning as in (A),
\[
\left|g_{A,\lambda}-\text{Id}\right|_{2;\mathbb{B}^{n}}\text{ is small}
\]
 hence
\[
\left|\theta-1\right|_{1;\mathbb{B}^{n}}\text{ is small}
\]
Then, we apply \cite[Theorem 5]{TE} and \cite[Lemma 10.4]{CDK} to
get a $C^{\infty}$ solution diffeomorphism to $\text{det\,}D\varphi=\theta$
with $\varphi=\text{Id in \ensuremath{\mathscr{C}}}$ and 
\[
\left|\varphi-\text{Id}\right|_{1;\mathbb{B}^{n}}\text{ small}
\]
It can be verified that in \cite[Theorem 5]{TE}, if the volume form
$\theta$ is smooth, the solution diffeomorphism $\varphi$ is also
smooth. This follows from the fact that the solution to the linearized
problem $\text{div}\,u=\theta-1$ in \cite[Theorem 3]{TE} is smooth
since it depends only on $\theta$ and not on $r,\,\alpha$ (see \cite[Remark 3 and Footnote 3]{TE})
and from the the way $\varphi$ is found (integrating the time dependent
vector field $u_{t}=u/((1-t)\theta+t$), c.f. \cite[Lemma 2]{DM},
\cite[p.209-210]{CDK}). One then uses the estimate in \cite[Theorem 3]{TE}
and that in \cite[Lemma 10.4]{CDK} to get the estimate $\text{\ensuremath{\left|\varphi-\text{Id}\right|_{1;\mathbb{B}^{n}}}}\leq C\left|\theta-1\right|_{1;\mathbb{B}^{n}}$,
for some constant $C=C(n)>0$. The construction then follows that
of case (\hyperlink{(A)}{A}). As shown in (D) below, the more general
(and abstract) result \cite[Lemma 10.4]{CDK} can actually entirely
replace the use of \cite[Theorem 5]{TE} above.

\medskip{}

\hypertarget{(C)}{}

\noindent \textbf{(C).} \noun{Linear dependence $\delta=\chi\epsilon_{0}$
for $0<\epsilon_{0}\leq1$ in the case $f\in C^{r,\alpha}\setminus C^{\infty}$.
}The case of $f\in C^{\infty}$ is similar, the changes needed being
indicated in (D) below. As in Section \hyperlink{se3.1}{3.1},\textbf{
}we shall establish a finite chain of linear bounds finally leading
to the determination of the constant $\chi$. We emphasize that $\left|\cdot\right|_{r,\alpha}$
in (C.2) - (C.5) is the $C^{r,\alpha}$ norm defined in (A) above
and $\left\Vert \cdot\right\Vert _{C^{1}}$ in (C.5) - (C.8) is the
standard Whitney $C^{1}$ norm (Section \hyperlink{se5.1}{5.1}) in
which Lemma \hyperlink{l2}{2} is formulated.

We start by establishing the actual estimate in Fact \hyperlink{F1}{1}.

\medskip{}

\noindent (\textbf{C.1})~\emph{~Given any} $n\in\mathbb{Z}^{+}$\emph{,}
$c\geq1$\emph{,} $A\in SL(n,\mathbb{R})$ \emph{and} $Df(0)\in SL_{c}$,
\[
\left\Vert A-Df(0)\right\Vert <\delta\leq1\Longrightarrow\left\Vert A^{-1}\circ Df(0)-\text{Id}\right\Vert <(c+1)^{n-1}\delta=C_{1}(c,n)\delta
\]
We have
\[
\|A^{-1}\circ Df(0)-\text{Id}\|=\|A^{-1}\circ(Df(0)-A)\|<\|A^{-1}\|\cdot\delta
\]
Since $A\in SL(n,\mathbb{R})$ and $\|A\|<c+1$, looking at its polar
decomposition one sees that
\[
\underset{x\in\mathbb{S}^{n-1}}{\text{min}}|A(x)|>(c+1)^{-n+1}
\]
thus $\|A^{-1}\|<(c+1)^{n-1}$ and the assertion follows.

\smallskip{}

In what follows, $C$, $C'$ and $C''$ denote auxiliary generic constants
(varying from step to step), whose existence follows from standard
Hölder estimates \cite[p.342 and 366]{CDK} or is evident from the
context.\medskip{}

\noindent (\textbf{C.2})~~$\left\Vert A^{-1}\circ Df(0)-\text{Id}\right\Vert <\delta\Longrightarrow\left|g_{A,\lambda}-\text{Id}\right|_{1,\alpha;\mathbb{B}^{n}}<C_{2}(n)\delta$.~
The partition function $\xi$ is fixed for each dimension $n$ and
$\left|\xi\right|_{1,\alpha;\mathbb{B}^{n}}\leq C(n)\left|\xi\right|_{2;\mathbb{B}^{n}}$
(\cite[p.342]{CDK}), therefore one has, by (4.7), for $\lambda$
small enough,
\[
\begin{array}{llll}
 & \left|g_{A,\lambda}-\text{Id}\right|_{1,\alpha;\mathbb{B}^{n}} & = & \left|\xi(h_{A,\lambda}-\text{Id})\right|_{1,\alpha;\mathbb{B}^{n}}\\
 &  & \leq & C'(n)\left|\xi\right|_{1,\alpha;\mathbb{B}^{n}}\left|h_{A,\lambda}-\text{Id}\right|_{1,\alpha;\mathbb{B}^{n}}\\
 &  & \leq & C''(n)\left|h_{A,\lambda}-\text{Id}\right|_{1,\alpha;\mathbb{B}^{n}}\\
 &  & < & C''(n)\delta=C_{2}(n)\delta
\end{array}
\]

\noindent (\textbf{C.3})~~$\left|g_{A,\lambda}-\text{Id}\right|_{1,\alpha;\mathbb{B}^{n}}<\delta\leq1\Longrightarrow\left|\theta-1\right|_{1,\alpha;\mathbb{B}^{n}}<C_{3}(n)\delta$.
~One has,
\[
\left|\theta-1\right|_{0,\alpha;\mathbb{B}^{n}}=\text{max}\left(\left|\theta-1\right|_{0;\mathbb{B}^{n}},\,\left[\theta\right]_{\alpha;\mathbb{B}^{n}}\right)
\]
Clearly $\left|\theta-1\right|_{0;\mathbb{B}^{n}}<C(n)\delta$ for
$\theta-1$ is the sum of $n!$ terms of the form
\[
{\color{blue}\pm}\big((\widehat{a}_{1}+\delta_{1})(\widehat{a}_{2}+\delta_{2})\cdots(\widehat{a}_{n}+\delta_{n})-\widehat{a}_{1}\widehat{a}_{2}\cdots\widehat{a}_{n}\big)
\]
where each $\widehat{a}_{i}=0$ or $1$ is an entry of the $\text{Id}$
matrix and $|\delta_{i}|<\delta\leq1$, thus $\left|\theta-1\right|{}_{0;\mathbb{B}^{n}}<n!(2^{n}-1)\delta$.

To simplify the notation, we write $g$ for the generic component
$g_{A,\lambda}^{i}$ of $g_{A,\lambda}$ and $\partial^{k}g$ for
its generic partial derivative of order $k$. 

In abridged notation, the determinant $\theta=\text{det}\,Dg_{A,\lambda}$
is the sum of $n!$ monomials of the form ${\color{blue}\pm}(\partial g)^{n}$.
Using the the following estimate for the $\alpha$-Hölder seminorm
of the product of scalar functions (\cite[p.366]{CDK}),
\[
\left[h_{1}\cdots h_{n}\right]_{\alpha}\leq n\,\underset{j}{\text{max}}\left|h_{j}\right|_{0}^{n-1}\cdot\underset{j}{\text{max}}\left[h_{j}\right]_{\alpha}
\]
and since by hypothesis 
\[
\underset{\mathbb{B}^{n}}{\text{sup}}|\partial g|\leq\left|g_{A,\lambda}\right|_{1,\alpha;\mathbb{B}^{n}}<\text{\ensuremath{\left|\text{Id}\right|_{1,\alpha;\mathbb{B}^{n}}}}+1=2
\]
 and $\left[\partial g\right]_{\alpha;\mathbb{B}^{n}}<\delta$, one
has (in abridged form)
\[
\left[\theta\right]_{\alpha;\mathbb{B}^{n}}\leq\sum^{n!}\left[(\partial g)^{n}\right]_{\alpha;\mathbb{B}^{n}}<n!n2^{n-1}\delta=C'(n)\delta.
\]
thus (C.3) holds.\medskip{}

\noindent (\textbf{C.4})~ Let $\varOmega=\mathbb{B}^{n}\,\setminus\frac{1}{4}\mathbb{D}^{n}$
and $U=(\mathbb{D}^{n}\setminus\frac{2}{3}\mathbb{B}^{n})\cup(\frac{1}{3}\mathbb{D}^{n}\setminus\frac{1}{4}\mathbb{B}^{n})$.
Let $\widehat{\epsilon}=\widehat{\epsilon}(r,\alpha,n)=\widehat{\epsilon}(r,\alpha,U,\varOmega)$
and $C_{4}=C_{4}(r,\alpha,n)=c(r,\alpha,U,\varOmega)$ be the corresponding
constants in \cite[Theorem 4]{TE}. One has for the solution diffeomorphism
$\varphi\in\text{Diff}^{r,\alpha}(\mathbb{B}^{n})$ obtained via \cite[Theorem 4]{TE}
in (A.2) above, 
\[
\left|\theta-1\right|_{0,\alpha;\mathbb{B}^{n}}<\delta\leq\epsilon(r,\alpha,n)\Longrightarrow\left|\varphi-\text{Id}\right|_{1;\mathbb{B}^{n}}<C_{4}(r,\alpha,n)\delta
\]

\noindent (\textbf{C.5})~~We now return to the Whitney $C^{1}$
norm. Since $\left\Vert \cdot\right\Vert _{C^{1}}\leq n\left|\cdot\right|_{1}$
for maps $\mathbb{B}^{n}\rightarrow\mathbb{R}^{n}$ (Section \hyperlink{se5.1}{5.1}),
one has
\[
\left|\varphi-\text{Id}\right|_{1;\mathbb{B}^{n}}<\delta\Longrightarrow\left\Vert \varphi-\text{Id}\right\Vert _{C^{1};\mathbb{B}^{n}}<n\delta=C_{5}(n)\delta
\]

\noindent (\textbf{C.6})~~Let $C_{6}=3$. Then, 
\[
\left\Vert \varphi-\text{Id}\right\Vert _{C^{1};\mathbb{B}^{n}}<\delta\leq1/2\Longrightarrow\left\Vert \varphi^{-1}-\text{Id}\right\Vert _{C^{1};\mathbb{B}^{n}}<C_{6}\delta.
\]
Since $\varphi^{-1}$ is a diffeomorphism of $\mathbb{B}^{n}$ onto
itself, one has
\[
\begin{array}{lll}
\left\Vert \varphi^{-1}-\text{Id}\right\Vert _{C^{1};\mathbb{B}^{n}}=\left\Vert (\varphi-\text{Id)}\circ\varphi^{-1}\right\Vert _{C^{1};\mathbb{B}^{n}} & \leq & \left\Vert \varphi-\text{Id}\right\Vert _{C^{1};\mathbb{B}^{n}}\big(1+\left\Vert \varphi^{-1}\right\Vert _{C^{1};\mathbb{B}^{n}}\big)\\
 & < & \delta\big(1+\left\Vert \varphi^{-1}\right\Vert _{C^{1};\mathbb{B}^{n}}\big)
\end{array}
\]
\[
\begin{array}{lll}
\left\Vert \varphi-\text{Id}\right\Vert _{C^{1};\mathbb{B}^{n}}<1/2 & \Longrightarrow & \underset{u\in\mathbb{S}^{n-1}}{\text{min}}|D\varphi(x;u)|>1/2\quad\forall x\in\mathbb{B}^{n}\\
 & \Longrightarrow & \underset{\mathbb{B}^{n}}{\text{sup}}\|D\varphi^{-1}\|<2
\end{array}
\]
therefore, as $\left\Vert \varphi^{-1}\right\Vert _{C^{0};\mathbb{B}^{n}}=1$
it follows that $\left\Vert \varphi^{-1}\right\Vert _{C^{1};\mathbb{B}^{n}}<2,$
thus (C.6) holds.

\medskip{}

\noindent (\textbf{C.7})~~Let $C_{7}=4$. Then, 
\[
\left\Vert \varphi^{-1}-\text{Id}\right\Vert _{C^{1};\mathbb{B}^{n}},\:\left\Vert g_{A,\lambda}-\text{Id}\right\Vert _{C^{1};\mathbb{B}^{n}}<\delta\leq1\implies\left\Vert g_{A,\lambda}\circ\varphi^{-1}-\text{Id}\right\Vert _{C^{1};\mathbb{B}^{n}}<C_{7}\delta
\]
Let $\widehat{g}:=g_{A,\lambda}$. Then
\[
\begin{array}{lll}
\left\Vert \widehat{g}\circ\varphi^{-1}-\text{Id}\right\Vert _{C^{1};\mathbb{B}^{n}} & \leq & \left\Vert (\widehat{g}-\text{Id})\circ\varphi^{-1}\right\Vert _{C^{1};\mathbb{B}^{n}}+\left\Vert \varphi^{-1}-\text{Id}\right\Vert _{C^{1};\mathbb{B}^{n}}\\
 & < & \left\Vert \widehat{g}-\text{Id}\right\Vert _{C^{1};\mathbb{B}^{n}}\big(1+\left\Vert \varphi^{-1}\right\Vert _{C^{1};\mathbb{B}^{n}}\big)+\delta\\
 & < & \delta\big(1+\left\Vert \text{Id}\right\Vert _{C^{1};\mathbb{B}^{n}}+1\big)+\delta=4\delta
\end{array}
\]

\noindent (\textbf{C.8})~~Let $C_{8}=C_{8}(c)=c+2$. Then, 
\[
\left\Vert A-Df(0)\right\Vert ,\,\left\Vert g_{A,\lambda}-\text{Id}\right\Vert _{C^{1};\mathbb{B}^{n}}<\delta\leq1\Longrightarrow\left\Vert A\circ g_{A,\lambda}-Df(0)\right\Vert _{C^{1};\mathbb{B}^{n}}<C_{8}\delta
\]
We use the following basic estimate: given any linear map $L\in L(n,\mathbb{R})$
and any $C^{1}$-bounded map $h:\mathbb{B}^{n}\rightarrow\mathbb{R}^{n}$,
\[
\left\Vert L\circ h\right\Vert _{C^{1};\mathbb{B}^{n}}\leq\|L\|\cdot\left\Vert h\right\Vert _{C^{1};\mathbb{B}^{n}}
\]
Now, writing $\widehat{g}$ for $g_{A,\lambda}$ and $D$ for $Df(0)$
\[
\begin{array}{lll}
\left\Vert A\circ\widehat{g}-D\right\Vert _{C^{1};\mathbb{B}^{n}} & = & \left\Vert (A-D)\circ\widehat{g}+D\circ(\widehat{g}-\text{Id})\right\Vert _{C^{1};\mathbb{B}^{n}}\\
 & \leq & \|A-D\|\cdot\left\Vert \widehat{g}\right\Vert _{C^{1};\mathbb{B}^{n}}+\|D\|\cdot\left\Vert \widehat{g}-\text{Id}\right\Vert _{C^{1};\mathbb{B}^{n}}\\
 & < & \delta(\left\Vert \text{Id}\right\Vert _{C^{1};\mathbb{B}^{n}}+1)+c\delta=(c+2)\delta
\end{array}
\]

\noindent (\textbf{C.9})~~Let $\widehat{\epsilon}=\widehat{\epsilon}(r,\alpha,n)$
be the constant obtained in (C.4). Note that we may assume that all
constants $C_{k}$ above are $\geq2$. Then, following the above chain
of linear estimates it is immediate to verify that the constant $\chi=\chi(r,\alpha,c,n)$
below satisfies the conclusions of Lemma \hyperlink{l2}{2} when $f\in C^{r,\alpha}\setminus C^{\infty}$:
\[
\chi=\frac{1}{C_{1}C_{2}C_{3}}\text{min}\,\Big(\widehat{\epsilon},\frac{1}{C_{4}C_{5}C_{6}C_{7}C_{8}}\Big)
\]

It remains only to verify that the $C^{r,\alpha}$ map $g_{A,\lambda}=\text{Id}+\xi(h_{A,\lambda}-\text{Id})$
is in fact a diffeomorphism. It is easily seen (see below) that 
\begin{equation}
\left\Vert g_{A,\lambda}-\text{Id}\right\Vert _{C^{1};\mathbb{B}^{n}}\leq1/4\implies g_{A,\lambda}\text{ is a diffeomorphism onto its image}
\end{equation}
Now, $\delta=\chi\epsilon_{0}\leq\chi$ since $\epsilon_{0}\leq1$,
therefore 
\[
\left\Vert g_{A,\lambda}-\text{Id}\right\Vert _{C^{1};\mathbb{B}^{n}}\leq n\left\Vert g_{A,\lambda}-\text{Id}\right\Vert _{1,\alpha;\mathbb{B}^{n}}\leq nC_{2}C_{1}\chi<{\textstyle 1/4}
\]
as $C_{5}=n$ and all constants $C_{k}$ are $\geq2$, thus
\[
nC_{2}C_{1}\chi\leq\frac{1}{C_{3}C_{4}C_{6}C_{7}C_{8}}\leq2^{-5}.
\]
It remains to prove (4.15): it is immediate from the hypothesis that
the derivative is everywhere nonsingular, thus only the injectivity
of $\widehat{g}:=g_{A,\lambda}$ needs to be established. We show
that for any $x,y\in\mathbb{B}^{n}$, $|\widehat{g}(y)-\widehat{g}(x)|\geq\frac{1}{6}|y-x|$.
The hypothesis implies that for any $v\in\mathbb{R}^{n}$, $|D\widehat{g}(0;v)|\geq\frac{2}{3}|v|$.
Let $h(x)=\widehat{g}(x)-\widehat{g}(0)-D\widehat{g}(0;x).$ Then,
\[
\underset{\mathbb{B}^{n}}{\text{sup}}\|Dh\|=\underset{\mathbb{B}^{n}}{\text{sup}}\|D\widehat{g}-D\widehat{g}(0)\|\leq\underset{\mathbb{B}^{n}}{\text{sup}}\|D\widehat{g}-\text{Id}\|+\underset{\mathbb{B}^{n}}{\text{sup}}\|\text{Id}-D\widehat{g}(0)\|\leq{\textstyle \frac{1}{4}}+{\textstyle \frac{1}{4}}={\textstyle \frac{1}{2}}
\]
thus, for any $x,\,y\in\mathbb{B}^{n}$, $|h(y)-h(x)|\leq\frac{1}{2}|y-x|$,
hence
\[
\begin{array}{lll}
|\widehat{g}(y)-\widehat{g}(x)| & = & |D\widehat{g}(0;y-x)+h(y)-h(x)|\\
 & \geq & |D\widehat{g}(0;y-x)|-|h(y)-h(x)|\\
 & \geq & \frac{2}{3}|y-x|-\frac{1}{2}|y-x|=\frac{1}{6}|y-x|
\end{array}
\]
Therefore $\widehat{g}=g_{A,\lambda}$ is injective and the proof
of (C) is complete.\medskip{}

\noindent \textbf{(D). }\noun{Linear dependence $\delta=\chi\epsilon_{0}$
for $0<\epsilon_{0}\leq1$ in the case $f\in C^{\infty}$.}\medskip{}

\noindent (\textbf{D.1})\textbf{~~}The estimate in (C.1) above carries
unchanged to the present $C^{\infty}$ case.\medskip{}

\noindent (\textbf{D.2})~~From (4.14), reasoning as in (C.2) now
applying the estimate for the $\left|\cdot\right|_{r}$ norm of the
product (end of Section \hyperlink{se5.1}{5.1}), we immediately get,
\[
\left\Vert A^{-1}\circ Df(0)-\text{Id}\right\Vert <\delta\Longrightarrow\left|g_{A,\lambda}-\text{Id}\right|_{2;\mathbb{B}^{n}}<2^{2}\left|\xi\right|_{2;\mathbb{B}^{n}}\delta=C_{2}(n)\delta
\]

\noindent (\textbf{D.3})\textbf{~~}$\left|g_{A,\lambda}-\text{Id}\right|_{2;\mathbb{B}^{n}}<\delta\leq1\Longrightarrow\left|\theta-1\right|_{1;\mathbb{B}^{n}}<C_{3}(n)\delta$\textbf{$.$}
The estimate $\left|\theta-1\right|_{0;\mathbb{B}^{n}}<n!(2^{n}-1)\delta$
was obtained in (C.3). In the abridged notation adopted there, the
components $\partial_{i}\theta=(\nabla\theta)^{i}$ of $\nabla\theta$
are of the form
\[
\sum^{n!}\sum^{n}{\color{blue}\pm}(\partial^{2}g)(\partial g)^{n-1}
\]
Since by hypothesis, $\text{sup}_{\mathbb{B}^{n}}|\partial g|<2$
and $\text{sup}_{\mathbb{B}^{n}}|\partial^{2}g|<\delta$ it follows
that 
\[
\text{\ensuremath{\underset{i}{\text{max}\,}}}\underset{\mathbb{B}^{n}}{\text{sup}}|\partial_{i}\theta|\leq n!n2^{n-1}\delta
\]
which together with the estimate above for $\left|\theta-1\right|_{0;\mathbb{B}^{n}}$
finally gives
\[
\left|\theta-1\right|_{1;\mathbb{B}^{n}}<n!n2^{n-1}\delta=C_{3}(n)\delta.
\]

\noindent (\textbf{D.4})\textbf{~~$\left|\theta-1\right|_{1;\mathbb{B}^{n}}<\delta\Longrightarrow\left|\theta-1\right|_{0,\frac{1}{2};\mathbb{B}^{n}}<\sqrt{2n}\delta=C_{4}(n)\delta$}.
Reasoning as in Section \hyperlink{se5.1}{5.1} (equivalence of norms
$\left|\cdot\right|_{r}$ and $\left\Vert \cdot\right\Vert _{C^{r}}$),
we have
\[
\underset{i}{\text{max}}\underset{\mathbb{B}^{n}}{\,\text{sup}}\left|\partial_{i}\theta\right|<\delta\Longrightarrow\underset{\mathbb{B}^{n}}{\text{sup}}\left\Vert \nabla\theta\right\Vert <\sqrt{n}\delta
\]
thus, by the mean value inequality,
\[
\begin{array}{lll}
\left[\theta-1\right]_{\frac{1}{2};\mathbb{B}^{n}} & = & \underset{x,y\in\mathbb{B}^{n};\,x\neq y}{\text{sup}}\frac{\,\left|\theta(y)-\theta(x)\right|\,}{\sqrt{\left|y-x\right|\,}}\\
 &  & \,\\
 & \leq & \underset{x,y\in\mathbb{B}^{n}}{\text{sup}}\sqrt{\left|y-x\right|\,}\underset{\mathbb{B}^{n}}{\text{\,sup}}\left\Vert \nabla\theta\right\Vert <\sqrt{2n}\delta
\end{array}
\]
Therefore, since $\left|\theta-1\right|_{0;\mathbb{B}^{n}}<\delta$,
(D.4) follows.

\medskip{}

\noindent (\textbf{D.5})~~$\left|\theta-1\right|_{0,\frac{1}{2};\mathbb{B}^{n}}<\delta\Longrightarrow\left|u\right|_{1;\mathbb{B}^{n}}<C_{5}(n)\delta$.
Let $u\in\mathfrak{X}^{\infty}(\mathbb{B}^{n})$ be the solution to
\[
\begin{cases}
\text{div}\,u=\theta-1\\
u=0 & \text{in }\mathscr{C}
\end{cases}
\]
obtained via \cite[Theorem 3]{TE} (see (A.2) for the meaning of $\mathscr{C}$
and (B) for the regularity of $u$), which satisfies
\[
\left|u\right|_{1;\mathbb{B}^{n}}\leq C(n)\left|\theta-1\right|_{0,\frac{1}{2};\mathbb{B}^{n}}<C(n)\delta
\]

\noindent (\textbf{D.6})\textbf{~~}$\left|\theta-1\right|_{1;\mathbb{B}^{n}},\,\left|u\right|_{1;\mathbb{B}^{n}}<\delta\leq1/2\Longrightarrow\left|\varphi-\text{Id}\right|_{1;\mathbb{B}^{n}}<C_{6}(n)\delta$.
For $t\in[0,1]$ let $f_{t}=(1-t)\theta+t$ and $u_{t}=u/f_{t}$.
Using \cite[Lemma 10.4]{CDK} with $\varOmega=\mathbb{B}^{n}$, $r=1$,
$\alpha=0$ and $T=1$, and since $u=0$ in $\mathscr{C}$, we obtain
a solution $\varphi:=\varphi_{1}\in\text{Diff}^{\infty}(\mathbb{B}^{n})$
to
\[
\begin{cases}
\text{det}\,D\varphi=\theta\\
\varphi=\text{Id} & \text{in \,}\mathscr{C}
\end{cases}
\]
(for the regularity of $\varphi$ see (B) above). Moreover (see below),
\begin{equation}
\left|\theta-1\right|_{1;\mathbb{B}^{n}},\quad\left|u\right|_{1;\mathbb{B}^{n}}<\delta\leq1/2\Longrightarrow\left|u_{t}\right|_{1;\mathbb{B}^{n}}<8\delta\leq4\quad\forall\,t\in[0,1]
\end{equation}
Therefore (still by \cite[Lemma 10.4]{CDK}),
\[
\left|\varphi-\text{Id}\right|_{1;\mathbb{B}^{n}}\leq C(n)\int_{0}^{1}\left|u_{t}\right|_{1;\mathbb{B}^{n}}\,dt\leq C(n)8\delta
\]
It remains to show that (4.16) holds:\smallskip{}

(0) $\text{max}_{t\in[0,1]}\left|u_{t}\right|_{0;\mathbb{B}^{n}}\leq2\left|u\right|_{0;\mathbb{B}^{n}}<2\delta$
since by hypothesis 
\[
\underset{t\in[0,1]}{\text{min}}\,\underset{\mathbb{B}^{n}}{\text{inf}}\,f_{t}>1/2
\]

(1) the partial derivatives of the components of $u_{t}$ are of the
form
\[
\partial_{j}u_{t}^{i}=\frac{(\partial_{j}u^{i})((1-t)\theta+t)-u^{i}(1-t)\partial_{j}\theta}{((1-t)\theta+t)^{2}}
\]
therefore
\[
\underset{i,j;\,t\in[0,1]}{\text{max}}\,\underset{\mathbb{B}^{n}}{\text{sup}}\left|\partial_{j}u_{t}^{i}\right|\leq{\textstyle (\frac{3\delta}{2}+\frac{\delta}{2})/\frac{1}{4}}=8\delta
\]
since $\left|\partial_{j}u^{i}\right|,\,\left|u^{i}\right|<\delta$,
$\text{max}_{t\in[0,1]}\text{sup}_{\mathbb{B}^{n}}\left|f_{t}\right|<3/2$,
$t\in[0,1]$ and $\text{sup}_{\mathbb{B}^{n}}\left|\partial_{j}\theta\right|<1/2$. 

\medskip{}

\noindent (\textbf{D.7})~~From this point onward the estimates are
the same as in (C.5) - (C.8) and accordingly we reindex the constants
$C_{5}$, $C_{6}$, $C_{7}$, $C_{8}$ there as $C_{7}$, $C_{8}$,
$C_{9}$, $C_{10}$, respectively. Again, we may assume that $C_{k}\geq2$
for $1\leq k\leq10$ and following the chain of estimates it is immediate
to verify that the constant $\chi=\chi(c,n)$ below satisfies the
conclusions of Lemma \hyperlink{l2}{2} when $f\in C^{\infty}$:
\[
\chi=\frac{1}{C_{1}C_{2}C_{3}\cdots C_{10}}
\]
Since $C_{3}>n$, reasoning as in (C.9) it is immediate to verify
that also in this case $g_{A,\lambda}$ is in fact a diffeomorphism
onto its image. The proof of Lemma \hyperlink{l2}{2}  is complete.
\end{proof}
\hypertarget{se5}{}

\section{Appendix}

\hypertarget{se5.1}{}

\subsection{$C^{r}$ norms of vector fields and maps.}

Let $\left|\cdot\right|$ be the Euclidean norm on $\mathbb{R}^{n}.$
Fix a (finite) regular $C^{\infty}$ atlas $(V_{j},\phi_{j})_{j\leq m}$
of $M$. Let $A\subset M$ be an open set and $X\in\mathfrak{X}^{r}(A)$\emph{
}a\emph{ }vector field of class $C^{r}$, $r\in\mathbb{Z}^{+},$ defined
on $A$\emph{. }On each (partial) local chart associated with $A$,
$(V_{j}\cap A,\phi_{j}|_{V_{j}\cap A})$, $X$ has an expression
\[
X_{j}:\,\phi_{j}(V_{j}\cap A)\longrightarrow\mathbb{R}^{n}
\]
$X$ is $C^{r}$-\emph{bounded} on $A$ (see Section \hyperlink{se2}{2})
if the Whitney $C^{r}$ norm of $X$ is finite:
\[
\begin{array}{c}
\left\Vert X\right\Vert _{C^{r};A}:=\underset{j;\,0\leq k\leq r}{\mbox{max}}\,\underset{\phi_{j}(V_{j}\cap A)}{\mbox{sup}}\left\Vert D^{k}X_{j}\right\Vert <\infty\end{array}
\]
As the atlas is regular, $C^{r}$ vector fields defined on $M$ are
always $C^{r}$-bounded. Here, $\Vert D^{0}X_{j}(x)\Vert:=|X_{j}(x)|$
and $\Vert D^{k}X_{j}(x)\Vert:=\mbox{max}_{u_{i}\in\mathbb{S}^{n-1}}|D^{k}X_{j}(x;\,u_{1},\ldots,u_{k})|$).
In Section \hyperlink{se3}{3} we work with the equivalent norm
\[
\left|X\right|_{r;A}:=\underset{\substack{i,j;\,0\leq|\sigma|\leq r}
}{\mbox{max}}\,\underset{\phi_{j}(V_{j}\cap A)}{\mbox{sup}}\left|\partial^{\sigma}X_{j}^{i}\right|
\]
where $X_{j}=(X_{j}^{1},\ldots,X_{j}^{n})$ and $\sigma$ runs over
all multiindices $\sigma=(\sigma_{1},\ldots,\sigma_{n})\in\mathbb{N}_{0}^{n}$
for which $\left|\sigma\right|=\sum\sigma_{i}\leq r$. It is easily
seen that
\[
\left|\cdot\right|_{r;A}\leq\left\Vert \cdot\right\Vert _{C^{r};A}\leq n^{(r+1)/2}\left|\cdot\right|_{r;A}
\]
noting that $\mbox{max}_{x\in\mathbb{S}^{n-1}}\sum_{i=1}^{n}\left|x_{i}\right|$
is attained when $\left|x_{1}\right|=\cdots=\left|x_{n}\right|=n^{-1/2}$,
thus implying that $\lambda\leq\Vert D^{k}X_{j}(x)\Vert\leq n^{(k+1)/2}\lambda$
for $\lambda=\mbox{max}_{i;\,|\sigma|=k}|\partial^{\sigma}X_{j}^{i}(x)|$.

With the obvious changes, the same definitions are adopted for the
$C^{r}$ norms of maps $X\in C^{r}(A;\mathbb{R}^{q})$ (the local
chart expressions of $X$ being then of the form $X_{j}=X\circ\phi_{i}^{-1}$),
provided we restrict to the subspace of those that are $C^{r}$ bounded\emph{.}
In this context, if $h\in C^{r}(A)$ and either $X\in\mathfrak{X}^{r}(A)$
or $X\in C^{r}(A;\mathbb{R}^{q}),$ then by Leibniz product rule,
\[
\left|hX\right|_{r;A}\leq2^{r}\left|h\right|_{r;A}\left|X\right|_{r;A}
\]
an inequality systematically used in Section \hyperlink{se3.1}{3.1}.

\hypertarget{se5.2}{}

\subsection{Local $C^{1}$-metrization of $\text{Diff}_{\mu}^{\,r,\alpha}(M)$
and chart representations }
\begin{defn}
(We recall the convention $C^{r,0}:=C^{r}$ and $C^{\infty,\alpha}:=C^{\infty}$).
Fix a conservative regular atlas $(V_{i},\phi_{i})_{i\leq m}$ of
$M$ as before (see the Convention, Section \hyperlink{se2}{2}).
Given $r\in\mathbb{Z}^{+}\cup\{\infty\}$, $0\leq\alpha\leq1$, $\text{Diff}_{\mu}^{\,r,\alpha}(M)$
is the group (under composition) of the $C^{r,\alpha}$ diffeomorphisms
$f$ of $M$ onto itself preserving the volume form, $\omega=f^{*}(\omega)$,
or equivalently, the Lebesgue measure $\mu$ induced by it on $M$.
These are the bijections $f:M\rightarrow M$ satisfying: for each
pair $i,j\leq m$,
\end{defn}
\begin{enumerate}
\item the map 
\[
f_{ji}=\phi_{j}\circ f\circ\phi_{i}^{-1}:\,\phi_{i}(V_{i}\cap f^{-1}(V_{j}))\longrightarrow\mathbb{R}^{n}
\]
is of class $C^{r,\alpha}$, and the same holds for $f^{-1}$ in place
of $f$;
\item $\text{det}\,Df_{ji}\equiv1$.
\end{enumerate}

\subsubsection{Covering system for $f\in\text{\emph{Diff}}_{\mu}^{\,r,\alpha}(M)$
and local $C^{1}$ metrization }

Given $f\in\text{Diff}_{\mu}^{\,r,\alpha}(M)$, by the compactness
of $M$ one can find a finite open cover $B_{l\leq\widetilde{m}}$
of $M$ and two maps
\[
i,\,j:\{1,\ldots,\widetilde{m}\}\longrightarrow\{1,\ldots,m\}
\]
such that
\[
\overline{B_{l}}\subset V_{i(l)}\text{ \,\,and \,\,}f(\overline{B_{l}})\subset V_{j(l)}
\]
The triple $B_{l\leq\widetilde{m}}$, $i$, $j$, is called a \emph{covering
system for} $f$ and will be denoted by $\varUpsilon$. For each $\epsilon>0$,
let $\mathcal{\mathscr{U}}_{\epsilon,\varUpsilon}(f)$ be the set
of those $g\in\text{Diff}_{\mu}^{\,r,\alpha}(M)$ such that for all
$l\leq\widetilde{m}$,
\[
g(\overline{B_{l}})\subset V_{j(l)}\quad\text{and}\quad\left\Vert g_{l}-f_{l}\right\Vert ,\;\left\Vert D(g_{l}-f_{l})\right\Vert <\epsilon
\]
where
\[
g_{l}=\phi_{j(l)}\circ g\circ\phi_{i(l)}^{-1}|_{B_{l*}}\qquad\text{and}\qquad B_{l*}:=\phi_{i(l)}(B_{l})
\]
, $f_{l}$ being defined in the same way. These $\mathcal{\mathscr{U}}_{\epsilon,\varUpsilon}(f)$
induce a $C^{1}$-topology on $\text{Diff}_{\mu}^{\,r,\alpha}(M)$
(see e.g. \cite[p.262]{PR}), making it locally metrizable by the
standard Whitney $C^{1}$ norm: for any $h,\,g\in\mathcal{\mathscr{U}}_{\epsilon,\varUpsilon}(f)$,
\[
d_{C^{1}}(h,g):=\left\Vert h-g\right\Vert _{C^{1}}=\underset{l\leq\widetilde{m}}{\text{max}}\left\Vert h_{l}-g_{l}\right\Vert _{C^{1}}
\]
Clearly, a covering system for $f$ also works for any $g\in\mathcal{\mathscr{U}}_{\epsilon,\varUpsilon}(f)$,
$\epsilon>0$.

\hypertarget{se5.2.2}{}

\subsubsection{Chart representations of $f|_{B}$.}

Given $f\in\text{Diff}_{\mu}^{\,r,\alpha}(M)$ suppose that $B\subset M$
is an open set such that $\overline{B}\subset V_{i}$ and $f(\overline{B})\subset V_{j}$
for some $i,j\leq m$. Then 
\[
\widehat{f}=f_{ji,B}=\left.\phi_{j}\circ f\circ\phi_{i}^{-1}\right|_{\phi_{i}(B)}
\]
is a \emph{chart representation of }$f|_{B}$ with domain $\phi_{i}(B)\subset\phi_{i}(V_{i})$
and target $\phi_{j}(V_{j})$. If $x\in B$ we call $\widehat{f}$
a \emph{chart representation of }$f$ \emph{around} $x$. To simplify
the notation, we abbreviate by $x$ the point $\phi_{i}(x)$ representing
$x$ in the domain of $\widehat{f}$. 

\hypertarget{se5.2.3}{}

\subsubsection{Comparable chart representations. }

Given any other $g\in\text{Diff}_{\mu}^{\,r,\alpha}(M)$ such that
$g(\overline{B})\subset V_{j}$, $\widehat{f}=f_{ji,B}$ and $\widehat{g}=g_{ji,B}$
are called \emph{comparable chart representations of $f$ and $g$
on $B$ }(alternatively, \emph{comparable chart representations of
}$f|_{B}$ \emph{and} $g|_{B}$). By the continuity of the composition
operator in relation to the $C^{1}$ norm, if $\|\widehat{f}-\widehat{g}\|_{C^{1}}$
is small then $\|\widetilde{f}-\widetilde{g}\|_{C^{1}}$ is small
for any other pair of comparable chart representations of $f|_{B}$
and $g|_{B}$. Thus, a $C^{1}$ perturbation of a chart representation
of $f|_{B}$ results in $C^{1}$ perturbations of all other chart
representations of $f|_{B}$, the transition between two such chart
representations being explicitly given by
\[
f_{\widehat{j}\widehat{i},B}=\phi_{\widehat{j}j}\circ f_{ji;B}\circ\phi_{i\widehat{i}}
\]
$\phi_{kl}=\phi_{k}\circ\phi_{l}^{-1}$ being the chart transition
maps.

\hypertarget{se5.3}{}

\subsection{Statement and proof of Lemma 3.}

\hypertarget{l3}{}
\begin{lem}
Let $M$ be a (second countable, Hausdorff) connected, boundaryless
$C^{\infty}$ $n$-manifold. Given a compact subset $K$ with an open
neighbourhood $U\subsetneq M$ such that $U\setminus K$ is connected,
there is a compact n-submanifold $V$ with connected $C^{\infty}$
boundary such that $K\subset\text{\emph{int}}\,V$ and \emph{$V\subset U$.}
\end{lem}
\begin{proof}
Take a finite cover $B_{1},\ldots,B_{j}$ of $K$ by open Euclidean
balls\footnote{$D\subset M$ is an Euclidean open ball if there is some local chart
($V_{i},\phi_{i})$ such that $\overline{D}\subset V_{i}$ and, up
to a translation, $\phi_{i}(D)=\lambda\mathbb{B}^{n}$ for some $\lambda>0$.} such that $V_{0}:=\cup_{i\leq j}\overline{B_{i}}\subset U$. Slightly
perturbing the $\overline{B_{i}}$'s if necessary, we can assume that
the smooth $(n-1)$-spheres $\partial\overline{B_{i}}$ intersect
transversely so that $V_{0}$ is a compact $n$-submanifold with piecewise
smooth boundary. Smooth out the ``edges'' of $V_{0}$ so that the
resulting $n$-submanifold $V_{1}$ has $C^{\infty}$ boundary and
still satisfies $K\subset\mbox{int}\,V_{1}$ and $V_{1}\subset U$
(this is clearly possible since the smoothing can be performed arbitrarily
near $\partial V_{0}$). Assume that $\partial V_{1}$ is disconnected
(otherwise we are done). The idea is to use the connectedness of $U\setminus K$
to connect successively and inside $U\setminus K$, all the components
of $\partial V_{1},$ thus creating a new submanifold satisfying the
desired conclusions. Needless to say, care must be taken to avoid
the intercrossing of the ``connecting tubes'', the nature of the
``connecting surgery'' depending, at each step $i$, on whether
the tube connecting two components of $\partial V_{i}$ is contained
in $V_{i}\setminus K$ or in $U\setminus\mbox{int}\,V_{i}$ (see below). 

There is no difficulty in showing that given any component $b_{0}$
of $\partial V_{1}$ there is a distinct component $b_{1}$ and an
injective $C^{\infty}$ path $\gamma:[0,1]\rightarrow U\setminus K$,
$\gamma'(t)\neq0$, such that
\[
\gamma(0)\in b_{0},\qquad\gamma(]0,1[)\cap\partial V_{1}=\emptyset,\qquad\gamma(1)\in b_{1}
\]
and $\gamma$ is transverse to $\partial V_{1}$ at $\gamma(0)$,
$\gamma(1)$. Clearly, $\gamma^{*}:=\gamma(]0,1[)$ is contained either
in (I) $(\mbox{int}\,V_{1})\setminus K$ or in (II) $U\setminus V_{1}$.
Thicken the embedded segment $\gamma([0,1])$ to a thin $C^{\infty}$
embedded ``tube'' $\mathbb{D}^{n-1}\times[0,1]\overset{f_{1}}{\hookrightarrow}U\setminus K$
with its bases $\mathbb{D}^{n-1}\times0$ and $\mathbb{D}^{n-1}\times1$
attached (respectively) to $b_{0}$ and $b_{1}$ so that:

\begin{enumerate}
\item the ``outer cylinder'' $\mathbb{S}^{n-2}\times[0,1]$ is smoothly
attached to $b_{0}$ and $b_{1}$;
\item as $\gamma^{*}$, $C=f_{1}(\mathbb{D}^{n-1}\times]0,1[)$ is disjoint
from $\partial V_{1}$.
\end{enumerate}
Now, as $\gamma^{*}$, $C$ is contained either in (I) or in (II).
In the first case let
\[
V_{2}=V_{1}\setminus f_{1}(\mathbb{B}^{n-1}\times[0,1])\qquad\mbox{("worm-hole drilling")}
\]
and in the second
\[
V_{2}=V_{1}\cup f_{1}(\mathbb{D}^{n-1}\times[0,1])\qquad\mbox{("solid handle attaching")}
\]
Since $V_{2}$ is obtained from $V_{1}$ modifying inside $U\setminus K$
only, it is immediate that $V_{2}$ is also an $n$-submanifold with
$C^{\infty}$ boundary still satisfying $K\subset\mbox{int}\,V_{2}$
and $V_{2}\subset U$, but $\partial V_{2}$ has one component less
than $\partial V_{1}$. If $\partial V_{2}$ is still disconnected,
then use a finite induction argument: we do with $V_{2}$ exactly
what was done with $V_{1},$ decreasing again the number of boundary
components by 1. After $k-1$ steps ($k=\,$number of components of
$\partial V_{1}$) we get a manifold $V=V_{k}$ as desired. 
\end{proof}

Centro de Matemática da Universidade do Porto \\
Rua do Campo Alegre, 687, 4169-007 Porto, Portugal{\small{}}\\
E-mail: \texttt{pteixeira.ir@gmail.com}\\

\end{document}